\documentclass[a4paper,11pt]{article}
\usepackage[english]{babel}
\usepackage{graphicx}

\usepackage{amsxtra}
\usepackage{amssymb}
\usepackage{amsmath}
\usepackage{amsthm}
\usepackage{color}
\usepackage{hyperref}

\newtheorem{prop}{Proposition}[section]

\newtheorem{assump}[prop]{Assumptions}
\newtheorem{lemma}[prop]{Lemma}
\newtheorem{theo}[prop]{Theorem}

\numberwithin{equation}{section}

\theoremstyle{remark}
\newtheorem{rmq}{Remark}

\newcommand{\e}{\epsilon}
\newcommand{\ve}{\varepsilon}

\newcommand{\la}{\langle}
\newcommand{\ra}{\rangle}
\newcommand{\di}{\displaystyle}

\newcommand{\n}{\nabla}

\newcommand{\C}{\mathbb{C}}
\newcommand{\R}{\mathbb{R}}
\newcommand{\N}{\mathbb{N}}
\newcommand{\Z}{\mathbb{Z}}

\newcommand{\MQ}{\mathbb{Q}}
\newcommand{\MP}{\mathbb{P}}
\title{On the time of existence of solutions of the 
Euler-Korteweg system}
\author{Corentin Audiard
\footnote{Sorbonne Universit\'es, UPMC Univ Paris 06, UMR 7598, Laboratoire Jacques-Louis Lions, 
F-75005, Paris, France }
\footnote{CNRS, UMR 7598, Laboratoire Jacques-Louis Lions, F-75005, Paris, France}
\footnote{Acknowledgement : the author thanks the 
ANR Project NABUCO ANR-17-CE40-0025 for financial support.}
}
\begin{document}
\maketitle

\begin{abstract}
Under a natural stability condition on the pressure, 
it is known that for small irrotational initial data,
the solutions of the Euler-Korteweg system are global in 
time when the space dimension is at least $3$. 
If the initial velocity has a small 
rotational part, we obtain a lower bound on the time 
of existence that depends only on the rotational part.
In the zero vorticity limit we recover the previous global
well-posedness result.
\\
Independently of this analysis, we also provide (in 
a special case) a simple example of solution that blows 
up in finite time.
\end{abstract}

\section{Introduction}
The Cauchy problem for the Euler-Korteweg system reads 
\begin{equation}\label{EK}
\left\{
\begin{array}{ll}
 \partial_t\rho +\text{div}(\rho u)=0,\\
 \partial_tu+u\cdot \n u+\n g(\rho)=\n \bigg(K\Delta \rho
 +\frac{1}{2}K'(\rho)|\n \rho|^2\bigg),\\
 (\rho,u)|_{t=0}=(\rho_0,u_0).
\end{array}
\right.\ (x,t)\in \R^d\times \R^+.
\end{equation}
$g$ is the pressure term, $K$ the capillary coefficient,
a smooth function $\R^{+*}\to \R^{+*}$.
It appears in the litterature in various contexts 
depending on $K$. $K$ constant has been largely 
investigated, see , and corresponds to capillary fluids.
The important case where $K$ is proportional to $1/\rho$ 
corresponds to the so called quantum fluids, the equations are then 
formally equivalent to the 
nonlinear Schr\"odinger equation 
\begin{equation}\label{NLS}
i\partial_t\psi+\Delta \psi=g(|\psi|^2)\psi,
\end{equation}
through the so called Madelung transform 
$\psi=\sqrt{\rho}e^{i\varphi},\ \nabla \varphi=u$. It is 
worth pointing out that even for a smooth solution of 
NLS the map $(\psi\to (\rho,u)$ is not well defined if 
$\psi$ cancels (existence of vortices).\\
The main result on local well-posedness for the general
Euler-Korteweg system is due to 
Benzoni-Danchin-Descombes \cite{Benzoni1}, we shall use 
the following (slightly simpler) version:
\begin{theo}[\cite{Benzoni1}]\label{locWP}
 For $(\rho_0-\alpha,u_0)\in \mathcal{H}^{s+1}(\R^d)$, $s>d/2+1$, with $\mathcal{H}^s:=H^{s+1}\times H^s$,
 there exists a unique solution
 $(\rho,u)\in (\alpha+C_tH^{s+1})\times C_tH^s$
to \eqref{EK}, and it exists on $[0,T]$ if the 
following two conditions are satisfied
\begin{enumerate}
 \item $\inf_{\R^d\times [0,T]} \rho(x,t)>0$,
 \item $\di \int_0^T \|\Delta \rho(s)\|_\infty+ 
 \|\n u(s)\|_\infty ds$.
\end{enumerate}
\end{theo}
The original proof also shows that the time of existence of 
the solutions is of order at least 
$\ln(1/\|(\rho_0-\alpha,u_0)\|_{H^{s+1}\times H^s})$. This rather 
small lower bound is due to the absence of assumptions on the 
pressure term which can cause exponentially growing instabilities. 
For stable pressure terms, this result was more recently sharpened 
by Benzoni and Chiron \cite{BenChir} who obtained the natural time 
$O(1/\|(\rho_0-\alpha,u_0)\|_{H^{s+1}\times H^s}$.\\
In irrotational settings, the author proved with B.Haspot \cite{AudHasp2} 
that small initial data lead to a global solution under 
a natural stability assumption on $g$.
The main focus of this paper is to describe more 
accurately the time of existence for small data that 
have a non zero rotational part.\\
We denote $\MQ=\Delta^{-1}\nabla\text{div}$ the 
projector on potential vector fields, $\MP=I-\MQ$
the projector on solenoidal vector fields.
In this paper, we prove the following informally stated
theorem (see theorems \ref{thd5} and \ref{thd3} for the precise statements):
\begin{theo}\label{maintheo}
Let $d\geq 3$, $\alpha>0$ a positive constant 
such that $g'(\alpha)>0$. For some 
function spaces $X,Y,Z$, if  
$\|\rho_0-\alpha\|_X,\ \|u_0\|_Y,\ \|\MP u_0\|_Z$ are 
small enough, then there exists $c(d,\alpha)>0$ such that the 
time of existence of the solution to \eqref{EK}
is bounded from below by $c/\|\MP u_0\|_Y$.
\end{theo}
Note that in the special case $\MP u_0=0$, we recover the global 
well-posedness result from \cite{AudHasp2}.\\
Before commenting the proof and sharpness of this result,
let us give a bit more background on the well-posedness 
theory of the Euler-Korteweg system.
\paragraph{Weak solutions} In the case of the 
quantum Navier-Stokes equations ($K$ proportional to 
$1/\rho$ and addition of a viscosity term) the 
existence of global weak solutions
has been obtained under various assumptions, an important  
breakthrough was obtained by Bresch et al \cite{BreschDesj}, 
introducing what is now called the Bresh-Desjardins entropy, a key 
a priori estimate to construct global weak solutions by 
compactness methods.\\
The inviscid case is more intricate. As the existence of global strong solutions to \eqref{NLS} with a large range 
of nonlinearities is well-known, Antonelli-Marcati 
\cite{AntMarc2} managed to use the formal equivalence with 
\eqref{EK} to construct global weak solutions, the main difficulty
difficulty being to give a meaning to the Madelung transform in 
the vacuum region where $\rho$ cancels, see also the review 
paper \cite{CDS} for a simpler proof. Relative entropy methods
have since been developed \cite{GiesTza},\cite{BreschGV} that
should eventually lead to the existence of global weak 
solutions for more general capillary coefficients $K$. Noticeably, these methods 
allow solutions with vorticity.
\paragraph{Strong solutions} As we mentioned, 
theorem \ref{locWP} is the first well-posedness
result in very general settings, an important idea due 
to Fr\'ed\'eric Coquel was to use a reformulation of the 
equations as a quasi-linear degenerate Schr\"odinger equation  
for which energy estimates in arbitrary high Sobolev spaces 
can be derived. \\
For quantum hydrodynamics ($K=1/\rho$) in the long wave regime
with irrotational speed,
the time interval of existence was improved by 
B\'ethuel-Danchin-Smets \cite{BDSm} thanks to the use 
of Strichartz estimates. This approach is not tractable to 
the general case of system \eqref{EK}. Note however that 
the second aim in \cite{BDSm} (long wave limit) was recently 
studied in \cite{BenChir} where the authors study 
\eqref{EK} in several long waves regimes and prove convergence 
to more classical equations such as Burgers, KdV or KP. Their 
analysis does not require the solutions to be irrotational.
\\
The analogy with the Schr\"odinger equation was 
pushed further in \cite{AudHasp2}
where the authors prove the existence of global strong 
solutions for small irotational data in dimension at 
least $3$. As a byproduct of the proof, such solutions 
behave asymptotically as solutions of the linearized system near a constant
density and zero speed, i.e. they ``scatter''.
The strategy of proof was reminiscent of ideas
developed by Gustavson, Nakanishi and Tsai \cite{GNT3}
for the Gross-Pitaevskii equation, and more generally 
the method of space time resonance (see Germain-Masmoudi-Shatah
\cite{GMS} for a clear description) which has had prolific
applications for nonlinear dispersive equations.
To some extent the present  paper is a continuation of 
such results for a mixed dispersive-transport system.
\paragraph{Travelling waves} The system \eqref{EK} being of 
dispersive nature, it is expectable that soliton like solutions 
exist, that is solutions that only depend on $x\cdot e-ct$ for 
some direction $e\in\R^d$ and speed $c$. In dimension $1$, the existence of solitons 
(traveling waves with same limits at $\pm \infty$) and kinks 
(different limits at $\pm \infty$) was derived in 
\cite{BDD2} by ODE methods. A stability criterion \`a la 
Grillakis-Shatah-Strauss \cite{GSS} was also exhibited. It is a stability of weak type, 
as it implies that the solution remains close to the soliton in 
a norm that does not give local well-posedness (stability ``until possible blow up''). 
Still in dimension $1$, the author proved the existence of multi-solitons 
type solutions, a first example of 
global solution in small dimension which is not an ODE solution.
Finally, motivated by the scattering result \cite{AudHasp2} in 
dimension larger than $2$, the author also proved in 
\cite{Audiard8} the existence of small amplitude traveling 
waves in dimension $2$, an obstruction to 
scattering.
\paragraph{Blow up} To the best of our knowledge, blow up for the Euler-Korteweg 
system is a completely open problem. The formation of vacuum for NLS equations 
with non zero conditions at infinity is also not clearly understood. We construct in 
appendix \ref{secblowup} a solution to \eqref{EK} (quantum case $K=1/\rho$) that 
blows up in finite time. The construction is very simple, it relies on
the existence of smooth solutions to \eqref{NLS} such that $\psi$ vanishes at some 
time and the reversibility of \eqref{EK}. 
\paragraph{The Euler-Korteweg system with a small 
vorticity} To give some intuition of our approach 
it is useful to introduce the reformulation 
from \cite{Benzoni1} : set $\n l:=w:=\sqrt{K/\rho}\n \rho$, 
then for smooth solution without vacuum \eqref{EK} is equivalent to the extended 
system
\begin{equation}\label{EEK}
\left\{
\begin{array}{lll}
\partial_tl+u\cdot \n l+a\text{div}u&=&0,\\
\partial_tw+\n(u\cdot w)+\n(a\text{div}u)&=&0,\\
\partial_tu+u\cdot \nabla u-w\cdot \n w-\n(a\text{div}w)
+g'w&=&0,
\end{array}\right.
\end{equation}
with $a=\sqrt{\rho K}$ and the second equation is simply the gradient of 
the first one.\\
If $u$ is irrotational, setting $z=u+iw$ we have 
using $\n\text{div}z=\Delta z$
\begin{equation*}
\partial_tz+ia\Delta z=\mathcal{N}(z),
\end{equation*}
despite the fact that $\mathcal{N}$ is a highly nonlinear 
term, the link with the Schr\"odinger equation is clear.
This observation is the starting point of the analysis
in \cite{}. Note that we have abusively neglected $g'(\rho)w$, which is at 
first order a linear term 
and thus must be taken into account for long time 
analysis. \\
If $u$ is not potential, it is natural to write 
$u=\MP u+\MQ u$ and split the potential and the solenoidal 
part of the last equation. For $v$ potential
$\di v\cdot\n v=\frac{1}{2}\n |v|^2$, so the last 
two equations of \eqref{EEK} rewrite
\begin{equation}\label{EEKvort}
\left\{
\begin{array}{lll}
 \partial_tw+\n(u\cdot w)+\n(a\text{div}\MQ u)&=&0,\\
 \partial_t\MQ u+\MQ(u\cdot \n\MP u+\MP u\cdot \n \MQ u
 )+\frac{1}{2}\n(|\MQ u|^2-|w|^2)-\n (a\text{div}w)+g'w&=&0,\\
 \partial_t\MP u+\MP(u\cdot\n \MP u+\MP u\cdot \n \MQ u)
 &=&0.
\end{array}
\right.
\end{equation}
The only important point is that the first two equations 
are still the same quasilinear Schr\"odinger equation
(where the Schr\"odinger evolution causes some decay), 
coupled to $\MP u$ and the evolution equation on 
$\MP u$ has $\MP u$ in factor of all its nonlinear 
terms.\\
A very simplified version of this 
dynamical system is the following ODE system
\begin{equation}\label{ODE}
\left\{
\begin{array}{ll}
 x'=-x+x^2+y^2,\\
 y'=y(x+y),
\end{array}
\right.
\end{equation}
where one should think of $x$ as $\MQ u+iw$, $y$ 
as $\MP u$ and the linear evolution $x'=-x$ gives decay.
The proof of the following elementary property is 
the guideline of this paper :
\begin{prop}\label{guide}
Assume $|x(0)|\leq \varepsilon,\ |y(0)|\leq \delta\leq 
\varepsilon$.
Then for $\varepsilon,\delta$ small enough there 
exists $c>0$ such that the solution of 
\eqref{ODE} exists on a time interval $[0,T]$ with 
$T\geq c/\delta$.
\end{prop}
\begin{proof}
We plug the ansatz 
\begin{equation}\label{ansatz}
|x(t)|\leq \delta+2\varepsilon e^{-t},
|y(t)|\leq 2\delta 
\end{equation}
in \eqref{ODE}: 
\begin{eqnarray*}
 |x(t)|&\leq& \varepsilon e^{-t}+\int_0^te^{s-t}
 \big(2\delta^2+8\varepsilon^2e^{-2s}+4\delta^2\big)ds
 \leq (\varepsilon+8\varepsilon^2)e^{-t}+6\delta^2,\\
|y(t)|&\leq& \delta+\int_0^t2\delta(\delta
+2\varepsilon e^{-s} +2\delta)ds\leq 
\delta(1+6\delta t+4\varepsilon )
\end{eqnarray*}
For $\varepsilon,\delta\leq 1/16$, $t\leq 1/(12\delta)$
we get 
\begin{equation*}
|x(t)|\leq \frac{3}{2}\varepsilon e^{-t}+\frac{3}{8}
\delta,\ |y(t)|\leq \frac{7}{4}\delta,
\end{equation*}
so that a standard continuation argument ensures that 
the solution exists on $[0,1/(12\delta)]$ and 
\eqref{ansatz} is true on this interval.
\end{proof}
Of course, some difficulties arise in our case: first due to the quasi-linear 
nature of the problem, loss of derivatives are bound to arise. This is handled 
by a method well-understood since the work of Klainerman-Ponce \cite{KlaiPonce}, 
where one mixes dispersive (decay) estimates with high order energy estimates
(see for example the introduction of \cite{AudHasp2} for a short description).\\
The second difficulty is more consequent and is due to some lack of integrability 
of the decay. Basically, we have $\|e^{it\Delta}\|_{L^p\to L^{p'}}
\lesssim 1/t^{d(1/2-1/p)}$, which is weaker as the dimension decreases. Again, 
it was identified in \cite{KlaiPonce} that this is not an issue for quasi-linear 
Schr\"odinger equations if $d\geq 5$, but the case $d<5$ requires much more 
intricate (and recent) methods.
\\
There has been an extremely abundant activity on 
global well-posedness for quasi-linear dispersive 
equations over  the last decade.  The method of space-time resonances 
initiated by Germain-Masmoudi-Shatah\cite{GMS} and 
Gustavson-Nakanishi-Tsai \cite{GNT3} led to numerous 
improvements and outstanding papers, a recent prominent 
result being the global well-posedness of the capillary-gravity
water waves in dimension $3$ due to Deng-Ionescu-Pausader-Pusateri
\cite{IonescuWWCG}.\\
The issue of long time existence for coupled 
dispersive-transport equations is more scarce. 
Nevertheless it arises naturally in numerous physical 
problem, and has been treated at least in the case 
of the Euler-Maxwell  system \cite{IonescuVort}.
The strategy of proof in this references seems to be 
close to proposition \ref{guide}, despite the 
considerable technical difficulties that are bound to 
arise.\\
It is worth pointing out that the time of existence is 
quite natural : it is related to the time of existence 
for $y'=y^2$, which is $1/y(0)$. It should be understood
that the finite time of existence is due to the 
lack of control of the transport equation.

\paragraph{Organization of the article}
We define our notations, functional framework and 
recall some technical tools in section \ref{secnot}.
Section \ref{secenergy} is devoted to some energy 
estimates for \eqref{EK}. The main $H^n$ energy estimate is a 
modification of the arguments in \cite{Benzoni1} and is proved for 
completeness in the appendix \ref{proofenergy}. 
As is common for dispersive equations, the proof of theorem  
\ref{maintheo} is more difficult in smaller dimensions. 
Here $d\geq 5$ is quite straightforward and is treated in 
section \ref{secd5} while $d=3$ is in section \ref{secd3}. $d=4$ is similar 
to $d=3$ but simpler,thus we do not detail this case. A large part of the analysis 
in dimension $3$ builds upon previous results from \cite{AudHasp2}, as such this 
part is not self-contained. The new difficulties are detailed, but 
the delicate estimates for the so-called ``purely dispersive'' quadratic nonlinearities is a bit 
redundant with \cite{AudHasp2} and are thus only partially 
carried out in the appendix \ref{spacetime}.
We also construct in section \ref{secblowup} an example of solution 
which blows up in finite time. This construction relies 
on the Madelung transform and the finite time formation
of vacuum for the Gross-Pitaevskii equation.

\section{Notations and functional spaces}
\label{secnot}
\paragraph{Constants and inequalities}
We will denote by $C$ a constant used in the 
bootstrap argument of sections \ref{secd5} and 
\ref{secd3}, it remains the same in the section.
Constants that are allowed to change from line 
to line are rather denoted $C_1,C_2\cdots$\\
We denote $a\lesssim b$ when there exists $C_1$ 
such that $a\leq C_1b$, with $C_1$ a ``constant'' 
that depends in a clear way on the various 
parameters of the problem.
\paragraph{Functional spaces}
$L^p(\R^d),W^{k,p}(\R^d), W^{k,2}=H^k(\R^d)$, $\dot{W}^{k,p}$ are the 
usual Lebesgue, Sobolev  and homogeneous Sobolev spaces. $L^{p,q}$ is the Lorentz space
obtained as an interpolation space of 
$L^{p_1},L^{p_2}$ by real interpolation with parameter
$q$, see \cite{BergLof}.\\
$\mathcal{S}(\R^d)$ is the Schwartz class, $\mathcal{S'}(\R^d)$ its dual, the 
space of tempered distribution.\\
If there is no ambiguity we drop the 
$(\R^d)$ reference. In our settings, $\rho$ is one derivative more regular than $u$, therefore we define 
$$\mathcal{H}^n=H^{n+1}\times H^n,\ 
\mathcal{W}^{k,p}=W^{k+1,p}\times W^{k,p}.$$
We recall the Sobolev embeddings
\begin{equation*}
\forall, kp<d,\ \dot{W}^{k,p}\hookrightarrow L^q,\ \frac{1}{q}=\frac{1}{p}-
\frac{k}{d},\ \forall\, kp>d,\ W^{k,p}\hookrightarrow C^0\cap L^\infty,
\end{equation*}
the tame product estimate for $p,q,r>1$, $\frac{1}{p}+\frac{1}{q}=\frac{1}{r}$
\begin{equation}
\label{productrule}
\|uv\|_{W^{k,r}}\lesssim \|u\|_{L^p}\|v\|_{W^{k,q}}+\|u\|_{W^{k,q}}\|v\|_{L^p},
\end{equation}
and the composition rule, for $F$ smooth, $F(0)=0$,
\begin{equation}
\forall\, u\in W^{k,p}\cap L^\infty,
\|F(u)\|_{W^{k,p}}\leq C(\|u\|_\infty)\|u\|_{W^{k,p}}\label{comprule}
\end{equation}

\paragraph{Fourier and bilinear Fourier multiplier}
The Fourier transform of $f\in \mathcal{S}'(\R^d)$ is denoted $\widehat{f}$
or $\mathcal{F}(f)$. A Fourier multiplier of symbol $m(\xi)$ with moderate 
growth acts on $\mathcal{S}$
\begin{equation*}
 m(D)f=\mathcal{F}^{-1}(m(\xi)\widehat{f}(\xi)),
\end{equation*}
with natural extensions for matrix valued symbols.
A multiplier denoted $m(-\Delta)$ is of course 
the multiplier of symbol $m(|\xi|^2)$.
\\
The Mihlin-H\"ormander theorem (see \cite{BergLof}) states that for $M$ large enough, if for any 
multi-index $\alpha$ with
$|\alpha|\leq M,\ |\nabla^km|\lesssim |\xi|^{-|\alpha|}$, then $m$
acts continuously on $L^p$, $1<p<\infty$.\\
A bilinear Fourier multiplier of symbol $B(\eta,\xi-\eta)$ acts on 
$\mathcal{S}^2$
\begin{equation*}
B[f,g]=\mathcal{F}^{-1}\bigg(\int_{\R^d}B(\eta,\xi-\eta)\widehat{f}(\eta)
\widehat{g}(\xi-\eta)d\eta\bigg)
=\mathcal{F}^{-1}\bigg(\int_{\R^d}B(\zeta-\eta,\zeta)\widehat{f}(\zeta-\eta)
\widehat{g}(\zeta)d\eta\bigg).
\end{equation*}
The Coifman-Meyer \cite{CoifmanMeyer2} theorem states that if $|\n^{k}B|\lesssim 1/(|\xi|+|\eta|)^k$
for sufficiently many $k$, then $B$ is continuous $L^p\times L^q\to L^r$, 
$1/p+1+q=1/r$, $p>1,\ q,r\leq \infty$. \\
We denote $\nabla_\xi B$ the bilinear multiplier of 
symbol $\n_\xi B(\eta,\xi-\eta)$, and similarly for 
$\n_\eta B$.

\paragraph{Potential and solenoidal fields}
Potential fields $v$ are vector fields of the form 
$v=\nabla f,\ f:\ \R^d\to \C$, they satisfy 
\begin{equation*}
\text{curl}(v)=(\partial_iv_j-\partial_jv_i)
_{1\leq i,j\leq d}=0.
\end{equation*}
Solenoidal fields satisfy $\text{div}(v)=\sum\partial_i 
v_i=0$.\\
The projector on potential vector fields is the Fourier 
multiplier $\mathbb{Q}=\Delta^{-1}\n \text{div}$, the 
projector on solenoidal vector fields is 
$\mathbb{P}=I_n-\mathbb{Q}$.
According to Mihlin-H\"ormander multiplier theorem, 
$\mathbb{P}=(-\Delta)^{-1}\n \text{div}$ and $\mathbb{Q}$ act continuously 
$L^p\to L^p$, $1<p<\infty$, and in the related Sobolev 
spaces.

\paragraph{Reformulation of the equations}
We denote $r=\rho-\alpha$, $r_0=\rho_0-\alpha$, $w:=\n l:=\sqrt{K/\rho}\n \rho$.
According to \cite{Benzoni1},  if $(r_0,u_0)\in \mathcal{H}^N$ with 
$N>d/2+1$, there exists a unique local 
solution to \eqref{EK} such that $(\rho-\alpha,u)\in C_t\mathcal{H}^N$. 
For $N$ large enough the solution is smooth so it is equivalent to 
work on the extended formulation \eqref{EEK}. 
\begin{assump}\label{assumptions}
Up to a change of variables, we can assume
\begin{enumerate}
 \item $\alpha=1$, 
\item $a(1)=1$, 
\item $g'(1)=2>0$.
\end{enumerate}
\end{assump}
Equations \eqref{EEKvort} read
\begin{equation}\label{EEKvort2}
\left\{
\begin{array}{lll}
 \partial_tw+\Delta \MQ u&=&
 \n \big((1-a)\text{div}\MQ u-u\cdot w\big),\\
 \partial_t\MQ u+\mathcal{N}(\MQ u,\MP u,w)
 +(-\Delta+2) w&=&
 \n \big((a-1)\text{div}w\big)+(2-g')w,\\
 \partial_t\MP u+\MP(u\cdot\n \MP u+\MP u\cdot \n \MQ u)
 &=&0.
\end{array}
\right.
\end{equation}
with 
$\mathcal{N}(\MQ u,\MP u,w)=\MQ(u\cdot \n\MP u
+\MP u\cdot \n \MQ u)+\frac{1}{2}\n(|\MQ u|^2-|w|^2)$.\\
Set $\di U=\sqrt{\frac{-\Delta}{2-\Delta}}$, $H=\sqrt{-\Delta(2-\Delta)}$,
then $\psi:=\MQ u+iU^{-1}w$ satisfies 
\begin{equation}\label{EEKschrod}
\left\{
\begin{array}{lll}
\partial_t\psi-iH\psi&=&\mathcal{N}_1(\psi,\MP u),
\\
\partial_t\MP u+\MP(u\cdot\n \MP u+\MP u\cdot \n \MQ u)
 &=&0.
\end{array}
\right.
\end{equation}
with 
\begin{equation}\label{nonlinearity}
\mathcal{N}_1= \n ((a-1)\text{div}w)+(2-g')w
+iU^{-1}\n \big((1-a)\text{div}\MQ u-u\cdot w\big)-\mathcal{N}.
\end{equation}
Note that $U^{-1}$ is singular, but we have for $1<p<\infty$
\begin{equation*}
\|U^{-1}w\|_{L^p}\sim \|\rho\|_{W^{1,p}},
\end{equation*}
therefore using the composition rule \eqref{comprule}, at least when 
$\|(\rho-1,u)\|_{W^{k,p}}<<1$ and $k$ is large enough
\begin{equation}
\|\psi\|_{W^{k,p}}\sim \|(\rho-1,u)\|_{\mathcal{W}^{k,p}}.
\end{equation}

\paragraph{Dispersive estimates}
Dispersion estimates for the semi-group $e^{itH}$ 
were obtained by Gustafson, Nakanishi and Tsai in 
\cite{GNT1}, a version in Lorentz spaces follows 
from real interpolation as pointed out 
in \cite{GNT3}.
\begin{theo}[\cite{GNT1}\cite{GNT3}]\label{dispersion}
For $2\leq p\leq \infty$, $s\in \R^+$, $U=\sqrt{-\Delta/(2-\Delta)}$, we have 
\begin{equation}\label{dispSobo}
\|e^{itH}\varphi\|_{W^{s,p}}\lesssim \frac{\|U^{(d-2)(1/2-1/p)}\varphi\|
_{W^{s,p'}}}{t^{d(1/2-1/p)}},
\end{equation}
and for $2\leq p<\infty$ 
\begin{equation}
\|e^{itH}\varphi\|_{L^{p,2}}\lesssim \frac{\|U^{(d-2)(1/2-1/p)}\varphi\|
_{L^{p',2}}}{t^{d(1/2-1/p)}} .
\end{equation}
\end{theo}
\begin{rmq}
 The estimates from \cite{GNT1} actually involve Besov spaces $B^s_{p,2}$ instead
 of $W^{s,p}$, and are slightly better than \eqref{dispSobo} due to the 
 embedding $B^s_{p,2}\subset W^{s,p}$, 
 $B^s_{p',2}\supset W^{s,p'}$ (see \cite{BergLof} chapter $6$).
\end{rmq}

\section{Energy estimates}
\label{secenergy}
\paragraph{High total energy estimate} 
The following energy estimate bounds all components 
of the solution $(\rho,u)$.
\begin{prop}\label{totenergy}
We recall the notation $r=\rho-1$. 
For $(r_0,u_0)\in \mathcal{H}^N$, $N$ 
large enough, $\|r\|_{W^{2,\infty}}$ small enough, 
\begin{equation*}
 \|(r,u)(t)\|_{\mathcal{H}^N}\leq
 C
 \|(r_0,u_0)\|_{\mathcal{H}^N}  
\text{exp}\big(\int_0^t 
C\|(r,u)\|_{\mathcal{W}^{1,\infty}}ds\big),
\end{equation*}
with $C=C(\|(r,u)\|_{L^\infty\mathcal{H}^N})$ a locally bounded function.
\end{prop}
The proof, not new, is postponed for completeness in 
appendix \ref{proofenergy}
\paragraph{Low transport energy estimate} 
\begin{prop}\label{energytransport}
Let $\MP u$ satisfy 
\begin{equation*}
\partial_t\MP u+\MP(u\cdot \n \MP u+\MP u\cdot \n\MQ u)
=0,
\end{equation*}
then for $p,q>1,\ k\in \N,\ 2k>d/q+1$ we have the a priori estimate 
\begin{equation}
\frac{d}{dt}\|\MP u\|_{W^{2k,p}}\lesssim (\|\MP u\|_
{W^{2k,q}}+\|\MQ u\|_{W^{2k,q}})\|\MP u\|_{W^{2k,p}}.
\end{equation}

\end{prop}

Energy estimates for transport type equations are
standard, see e.g. the textbook \cite{BCD} chapter 
$3$. Since the ``transport'' term is 
$\MP (u\cdot\n \MP u)$ rather than $u\cdot \n \MP u$, 
we include a short self-contained proof.
\begin{proof}
 Set $P_k=\Delta^k\MP u$, then $\Delta^k \MP=\Delta^k
 -\Delta^{k-1}\n \text{div}$ is a differential operator
 of order $2k$ so that 
 \begin{equation*}
\partial_t P_k+(u\cdot \n P_k)=R_k(\MP u,\MQ u),
 \end{equation*}
Note that since  $\Delta^k\MP \MQ=0$ we have 
$$R_k=-[\Delta^k\MP,u\cdot  \n]\MP u-\Delta^k\MP(\MP u\cdot \n \MQ u)=
-[\Delta^k\MP,u\cdot  \n]u-[\Delta^k\MP,\MP u\cdot \n] \MQ u.$$
We take the scalar product with $|P_k|^{p-2}P_k$ and integrate in space to get 
\begin{equation*}
\frac{d}{dt}\|P_k\|_p^p\lesssim \|\text{div}(u)\|_\infty\|P_k\|_p
+\|R_k|P_k|^{p-2}P_k\|_1.
\end{equation*}
Since $W^{2k,q}\subset W^{1,\infty}$, we are left to estimate terms of the 
form $\|\partial^\alpha \MP u\partial^\beta
v|P_k|^{p-1}\|_1$ with $v$ a placeholder for $\MP u$ or $\MQ u$, 
$|\alpha|+|\beta|=2k+1$. For $1/p_1+1/p_2=1/p$
we have 
\begin{eqnarray*}
\|\partial^\alpha \MP u\partial^\beta v|P_k|^{p-1}\|_1&\lesssim &
\|\MP u\|_{W^{|\alpha|,p_1}}\|v\|_{W^{|\beta|,p_2}}\|\MP u\|_{W^{2k,p}}^{p-1}
\\
&\lesssim& 
\|\MP u\|_{W^{2k,p}}\|v\|_{W^{|\beta|,p_2}}\|\MP u\|_{W^{2k,p}}^{p-1},
\end{eqnarray*}
provided $1/p-(2k-|\alpha|)/d\leq 1/p_1\leq 1/p$, which is equivalent to 
\begin{equation*}
0\leq \frac{1}{p_2}\leq \frac{2k-|\alpha|}{d}=\frac{|\beta|-1}{d}.
\end{equation*}
On the other hand the condition $W^{2k,q}\subset W^{|\beta|,p_2}$ is satisfied 
provided 
$\di \frac{1}{q}-\frac{2k-|\beta|}{d}\leq \frac{1}{p_2}$, the two conditions 
on $p_2$ lead to $1/q<(2k-1)/d$ which is the assumption.
We conclude
\begin{equation}\label{enhaut}
\frac{d}{dt}\|P_k\|_p^p\lesssim 
(\|\MP u\|_{W^{2k,q}}+\|\MQ u\|_{W^{2k,q}})\|\MP u\|_{W^{2k,p}}^{p}.
\end{equation}
Taking the $L^p$ norm in \eqref{transport} and using the continuity of 
$\MP:\  L^p\to L^p$ directly gives 
\begin{equation}\label{enbas}
 \frac{d}{dt}\|\MP u\|_p\lesssim (\|\MQ u\|_\infty +\|\MP u\|_\infty)
 \|\MP u\|_{W^{1,p}}
 \lesssim (\|\MQ u\|_{W^{2k,q}}+\|\MP u\|_{W^{2k,q}})
 \|\MP u\|_{W^{2k,p}}.
\end{equation}
Summing \eqref{enhaut} and \eqref{enbas} concludes.
\end{proof}

\section{Well-posedness for \texorpdfstring{$d\geq 5$}{d5}}\label{secd5}
The main result of this section is the following :
\begin{theo}\label{thd5}
Under assumptions \ref{assumptions}, 
for $d\geq 5$, there exists $(\varepsilon_0,c,N,k)\in 
(\R^{+*})^2\times \N^2$ such that for 
$\varepsilon\leq \varepsilon_0$, $\delta <<\varepsilon$,
if
$$\|(\rho_0-\alpha,u_0)\|_{\mathcal{H}^N\cap 
\mathcal{W}^{k,4/3}}\leq \varepsilon,\
\|\MP u_0\|_{W^{k,4}}\leq \delta,$$
then the solution of \eqref{EK} exists on $[0,T]$ with 
\begin{equation*}
T\geq \frac{c}{\delta}.
\end{equation*}
\end{theo}
\noindent
We recall that the system satisfied by $\psi=\MQ u+iU^{-1}w$ and 
$\MP u$ is (see \eqref{EEKschrod})
\begin{equation}
\left\{
\begin{array}{lll}
\partial_t\psi-iH\psi&=&\mathcal{N}_1(\psi,\MP u),
\\
\partial_t\MP u+\MP(u\cdot\n \MP u+\MP u\cdot \n \MQ u)
 &=&0.
\end{array}
\right.
\end{equation}
We will prove a priori estimates for the 
solution in a space where local well-posedness holds.

\paragraph{The bootstrap argument } We shall prove 
the following property : for $c,\varepsilon$ small enough, there exists 
$C>0$ such that for $t\leq c/\delta$, 
if we have the estimates 
\begin{equation*}
\|\psi\|_{H^N}+\|\MP u\|_{H^N}\leq C \varepsilon,\ 
\|\psi\|_{W^{k,4}}\leq C\delta
+\frac{C\varepsilon}{(1+t)^{d/4}},
\ \|\MP u\|_{W^{k,4}}\leq C\delta,
\end{equation*}
that we respectively name  total energy, dispersive estimate and transport energy, then 
\begin{equation*}
\|\psi\|_{H^N}+\|\MP u\|_{H^N}\leq C \varepsilon/2,\ 
\|\psi\|_{W^{k,4}}\leq C\delta/2
+\frac{C\varepsilon}{2(1+t)^{d/4}},
\ \|\MP u\|_{W^{k,4}}\leq C\delta/2.
\end{equation*}
From now on, $C$ is only used for the constant of the 
bootstrap argument, while other constants are labelled
as $C_1,C_2$... and can change from line to line.
\paragraph{The energy estimate}
Since $\|\psi\|_{H^N}\sim \|(\rho-1,\MQ u)\|_{\mathcal{H}^N}$,
the energy estimate of proposition \ref{totenergy} 
implies for $k> d/4+1$
\begin{eqnarray*}
\|\psi\|_{H^N}+\|\MP u\|_{H^N}&\leq& C_1\|z_0\|_{H^N}\text{exp}
\bigg(C_2\int_0^t\|\psi\|_{W^{k,4}}+\|\MP u\|_{W^{k,4}}ds\bigg)\\
&\leq &
C_1\varepsilon\text{exp}(C_2C\big(2\delta t
+\varepsilon/(d/4-1))\big).
\end{eqnarray*}
Take $C\geq 2C_1e^1$, for $t\leq c\delta$, 
$\varepsilon,c$ small enough (depending on $C$) 
we have 
\begin{equation}\label{energy}
\|\psi\|_{H^N}+\|\MP u\|_{H^N}\leq C_1e^1\varepsilon\leq 
C\varepsilon/2.
\end{equation}
\paragraph{The transport energy estimate} 
We apply Proposition \ref{energytransport} with 
$p=q=4$, $k$ even, $4k>d$, $t\leq c/\delta$
\begin{eqnarray}\nonumber
\frac{d}{dt}\|\MP u\|_{W^{k,4}}&\lesssim &
(\|\MP u\|_{W^{k,4}}+\|\MQ u\|_{W^{k,4}})
\|\MP u\|_{W^{k,4}}\leq \bigg(C\delta +C\delta+
\frac{C\varepsilon}{(1+t)^{d/4}}\bigg)C\delta\\
\Rightarrow \|\MP u\|_{W^{k,4}}&\leq &
\delta\bigg(1+2C_1C^2c
+\frac{C^2C_1\varepsilon}{d/4-1}\bigg)
\leq 2\delta<C\delta/2, \label{transport}
\end{eqnarray}
for $c,\varepsilon$ small enough, $C>2$.
\paragraph{The dispersive estimate}
The first equation in \eqref{EEKschrod} rewrites 
\begin{equation*}
\psi(t)=e^{itH}\psi_0+\int_0^te^{i(t-s)H}\mathcal{N}_1(\psi,\MP u)ds,
\end{equation*}
The linear evolution $e^{itH}\psi_0$ is estimated with 
the dispersive estimate \eqref{dispSobo} and Sobolev embeddings 
\begin{equation}\label{estimlin}
\|e^{itH}\psi_0\|_{W^{k,4}}\lesssim
\min(\|\psi_0\|_{W^{k,4/3}}/t^{d/4},\|\psi_0\|_{H^{k+d/4}})
\lesssim \frac{\varepsilon}{(1+t)^{d/4}}.
\end{equation}
The structure of the nonlinearity does not matter here, 
the only important points are 
\begin{enumerate}
 \item The presence of $U^{-1}$ in 
 $U^{-1}\n\big((1-a)\text{div}\MQ u)\big)$ is not an 
 issue since $U^{-1}\n=\sqrt{2-\Delta}\n/|\n$ is the 
 composition of a smooth Fourier multiplier and the 
 Riesz multiplier,
 \item All nonlinear terms are at least quadratic, 
 and involve derivatives of order at most $2$.
\end{enumerate}
We only detail the estimate of $\MQ (\MQ u\cdot\n \MP 
u)$ as the others can be done in a similar (simpler) 
way. Using the dispersion estimate and Sobolev 
embedding
\begin{eqnarray*}
 \bigg\|\int_0^te^{i(t-s)H} \MQ(\MQ u\cdot \n \MP 
 u)ds\bigg\|_{W^{k,4}}
 &\lesssim& \int_0^{t-1}
 \frac{\|\MQ u\cdot \n \MP u\|_{W^{k,4/3}}}{(t-s)^{d/4}}ds+\int_{t-1}^t\|\MQ u\cdot \n\MP u\|_{H^{k+
 d/4}}ds.
\end{eqnarray*}
The product rules give 
\begin{eqnarray*}
\|\MQ u\cdot \n \MP u\|_{W^{k,4/3}}&\lesssim& 
\|\MQ u\|_{L^4}\|\MP u\|_{H^{k+1}}+\|\MQ u\|_{W^{k,4}}
\|\MP u\|_{H^1}\leq  2C^2\bigg(\delta 
+\frac{\varepsilon}{(1+s)^{d/4}}\bigg)\varepsilon,\\
\|\MQ u\cdot \n \MP u\|_{H^{k+d/4}}&\lesssim &
\|\MQ u\|_{W^{k+d/4,4}}\|\MP u\|_{W^{1,4}}
+\|\MQ u\|_{L^4}\|\MP u\|_{W^{k+1+d/4,4}}\\
&\lesssim& \|\MQ u\|_{H^N}\|\MP u\|_{W^{k,4}}
+\|\MQ u\|_{W^{k,4}}\|\MP u\|_{H^N}\\
&\leq& C^2\varepsilon\delta
+C^2\varepsilon\bigg(\delta+\frac{\varepsilon}{(1+s)^{d/4}}\bigg).
\end{eqnarray*}
The bootstrap assumption directly gives 
\begin{eqnarray*}
 \bigg\|\int_0^te^{i(t-s)H} \MQ(\MQ u\cdot \n \MP 
 u)ds\bigg\|_{W^{k,4}}&\leq &
 C_1C^2\bigg(\delta \varepsilon+\varepsilon^2
 \int_0^{t-1}\frac{1}{(1+s)^{d/4}(t-s)^{d/4}}ds\bigg)\\
 && +C_1C^2\varepsilon \bigg(\delta
 +\frac{\varepsilon}{(1+t)^{d/4}}\bigg) \\
 &\leq& C_2C^2\varepsilon\bigg(\delta+
 \frac{\varepsilon}{(1+t)^{d/4}}\bigg).
 \end{eqnarray*}
We conclude by using \eqref{estimlin}, for 
$C$ large enough, $\varepsilon$ small enough
\begin{equation}\label{bootdispersion}
\|z(t)\|_{W^{k,4}}\leq \frac{C_0\varepsilon}{(1+t)^{d/4}}
+C_1 C^2\varepsilon \bigg(\delta+
 \frac{\varepsilon}{(1+t)^{d/4}}\bigg)\leq \frac{C}{2}
 \bigg(\delta+\frac{\varepsilon}{(1+t)^{d/4}}\bigg).
\end{equation}
\paragraph{End of proof}
Putting together \eqref{energy}, \eqref{transport} 
and \eqref{bootdispersion}, we see that as long as the 
solution exists and $t\leq c/\delta$, $\|z\|_{\mathcal{H}^N}$ 
remains  small and $\rho$ remains 
bounded away from $0$. According to the blow up 
criterion the solution exists at least for 
$t\leq c/\delta$.
\section{Well-posedness for \texorpdfstring{$d=3,4$}{d34}}
\label{secd3}
This section is similar to the previous one 
but is significantly more technical.
The low dimension version of theorem \ref{thd5} reads
\begin{theo}\label{thd3}
Under assumptions \ref{assumptions}, 
for $d=3,4$, there exists $(\varepsilon,c,N,k)\in 
(\R^{+*})^2\times \N^2$, $p>\frac{2d}{d-2}$
such that for $\delta <<\varepsilon$, if
$$\|(r,u_0)\|_{\mathcal{H}^N\cap 
\mathcal{W}^{k,p'}}+\||x|(r_0,\MQ u_0)\|_{L^2}
\leq \varepsilon,\
\|\MP u_0\|_{W^{k,p'}\cap W^{k,p}}+
\||x|\MP u_0\|_{L^2}\leq \delta,$$
then the solution of \eqref{EK} exists on $[0,T]$ with 
\begin{equation*}
T\geq \frac{c}{\delta}.
\end{equation*}
\end{theo}
\begin{rmq}
Unlike $d\geq 5$, one can not directly use the 
dispersive estimate to get closed bounds. This 
approach works for cubic and higher order 
nonlinearities, but not for quadratic terms. Therefore the emphasis
is put here on how to control quadratic terms, while the analysis of higher
order terms is much less detailed. We label such terms as ``\textit{cubic}''
and they are generically denoted $R$. The fact that they include loss of 
derivatives is unimportant. 
\end{rmq}
\noindent
For $\psi:\ [0,T]\times \R^d\to \C^d$, and $C$ 
a constant to choose later, we use the 
following notations:
\begin{eqnarray*}
\|\psi\|_{X(t)}&=&\max\big(\|\psi(t)\|_{H^N}+\|xe^{-itH}\psi\|_{L^2},\ (1+t)^{3(1/2-1/p)}(\|\psi\|_{W^{k,p}}-C\delta)\big),\\
\|\psi\|_{X_T}&=&\sup_{[0,T]}\|\psi\|_{X(t)}.
\end{eqnarray*}
For simplicity of notations, we only consider the (most difficult) case 
$d=3$. 
\subsection{Preparation of the equations}
We recall that the extended system is 
\begin{equation}
\left\{
\begin{array}{lll}
\partial_t\psi-iH\psi&=&\mathcal{N}_1(\psi,\MP u)+R,\text{ R cubic},
\\
\partial_t\MP u+\MP(u\cdot\n \MP u+\MP u\cdot \n \MQ u)
 &=&0,
\end{array}
\right.
\end{equation}
\begin{eqnarray*}
\mathcal{N}_1= \n ((1-a)\text{div}w)+(2-g')w
-\frac{1}{2}\n(|\MQ u|^2-|w|^2)
+iU^{-1}\n \big((1-a)\text{div}\MQ u-u\cdot w\big)
\\
-\MQ(u\cdot \n\MP u+\MP u\cdot \n \MQ u),
\end{eqnarray*}
the first line of the nonlinearity $\mathcal{N}_1$
depends only on the dispersive variable 
$\psi$ (``purely dispersive terms'') 
while the second line contains interaction between $\psi$ and the transport component $\MP u$ (``dispersive-transport terms'').\\
In order to apply the method of space-time resonances, it is useful  
that the Fourier transform of the purely dispersive nonlinear terms cancels 
at $0$. As such, the real part $\n ((1-a)\text{div}w)+(2-g')w
+\frac{1}{2}\n(|\MQ u|^2-|w|^2)$ is well prepared, but not the 
imaginary part $U^{-1}\n ((1-a)\text{div}\MQ u)$. We refer to the 
discussion at the beginning of section  $5$ in \cite{AudHasp2} for a 
more detailed motivation.\\
As in \cite{AudHasp2} (see also \cite{GNT3}) we use the following normal 
form transform:
\begin{lemma} \label{formenorm} 
For
\begin{equation*}
w_1=w-\n \big(B[w,w]-B[\MQ u,\MQ u]\big).
\end{equation*}
with $B$ the bilinear Fourier multiplier of 
symbol $\di \frac{a'(1)-1}{2(2+|\eta|^2)
+|\xi-\eta|^2)}$. Then $w_1$ satisfies 
\begin{equation}\label{eqw1}
\partial_tw_1+\Delta \MQ u=
\n \text{div}((1-a)\MQ u)+R,
\end{equation}
where $R$ contains cubic and higher order 
nonlinearities in $\MQ u,\MP u, l$.\\
Moreover, for any $T>0$ the map $\psi=\MQ u+iU^{-1}w\to \MQ u+iU^{-1}w_1$ 
is bi-lipschitz on a neighbourhood of $0$ in $X_T$, 
it is also bi-Lipschitz near $0$ for the norm  
$\|\psi_0\|_{H^N\cap W^{k,p'}}+\||x|\psi_0\|_{L^2}$.
\end{lemma}
\begin{proof}
According to \eqref{EEKvort2} $w$ satisfies 
\begin{eqnarray*}
\partial_tw+\Delta \MQ u&=&\n \big((1-a)\text{div}\MQ u\big)
-\n(u\cdot w)\\
&=&\n\text{div}\big((1-a)\MQ u\big)
+\n\big(\n a \cdot \MQ u\big) -\n(\MQ u\cdot w)-\n(\MP u\cdot w)\\
&=&\n\text{div}\big((1-a)\MQ u\big)+\n\big( (a'(1)-1)w\cdot \MQ u\big)
-\n (\MP u\cdot w)+R,
\end{eqnarray*}
with $R=\n\big((\n a-a'(1)w)\cdot \MQ u\big)$ a cubic term. Then $w_1$
satisfies
\begin{eqnarray*}
\partial_tw_1+\Delta \MQ u&=&\n\text{div}\big((1-a)\MQ u\big)
-\n (\MP u\cdot w)+R\\
&& +\n\bigg((a'(1)-1)w\cdot \MQ u+2B[w,\Delta \MQ u]
+2B[(\Delta-2)w,\MQ u]\bigg)\\
&=&\n\text{div}\big((1-a)\MQ u\big)
-\n (\MP u\cdot w)+R,
\end{eqnarray*}
by construction of $B$ (note that $R$ 
includes now terms like 
$\n B[\MQ (u\cdot \n \MP u),\MQ u]$ that have all a 
gradient in factor). 
The fact that $w\to w_1$ is bi-Lipschitz 
is proposition $5.4$ and proposition $5.5$ 
in \cite{AudHasp2}.
\end{proof}

\paragraph{Final form of the equations}
We define $b(\MQ u,w)=B[w,w]-B[\MQ u,\MQ u]$ so that $w_1=w-\n b(\MQ u,w)=
w-\n b(\MQ u,w_1)+R$, $R$ cubic. The new system on $\Psi=\MQ u+iU^{-1}w_1$ and $\MP u$ is 
\begin{equation}
\left\{
\begin{array}{lll}
\partial_t\Psi-iH \Psi&=&\n\bigg((\Delta-2)b+
(a-1)\text{div}w_1
-\frac{1}{2}\big(|\MQ u|^2-|w_1|^2\big)\bigg)
+(2-g')w_1
\\
&&+iU^{-1}\n \text{div}\big((1-a)\MQ u\big)\\
\label{equfinal}
&&-iU^{-1}\n (\MP u\cdot w_1)-\MQ(u\cdot \n \MP u+\MP u\cdot \n \MQ u)+R,\\
\partial_t\MP u&=&-\MP (u\cdot \n \MP u+\MP u\cdot \n \MQ u),
\end{array}
\right.
\end{equation}
with $R$ containing cubic  terms.
\begin{rmq}\label{rmqcubic}
Note that all cubic terms in \eqref{eqw1}
are gradients of the unknowns (see also the system
\eqref{EEKvort2}), therefore 
the change of variables $w_1\to U^{-1}w_1$ creates 
no nonlinearities with singularity at low frequency. For example, the new term 
$\n B[\MQ (u\cdot \n \MP u),\MQ u]$ becomes 
$U^{-1}\n B[\MQ(u\cdot \n \MP u),\MQ u]$.
\end{rmq}

Note that $(1-a)\MQ u$ is a quadratic term since at main order it is 
$-a'(1)l\MQ u$, with $l=\Delta^{-1}\text{div}w$.
\begin{rmq}\label{remarkequivnorm}
An important consequence of lemma \ref{formenorm} is 
that it suffices to estimate $\Psi$ instead of 
$\psi$, and the smallness of $\psi_0$
implies the smallness of $\Psi_0$. 
\end{rmq}
According to the remark above, it is sufficient to 
prove the following :
\begin{theo}
Under assumptions $2.1$, there exists $\varepsilon,
c,N,k\in (\R^{+*})^2\times \N^2,\ p>2d/(d-2)$ such that 
for $\delta<<\varepsilon$, if 
\begin{equation*}
\|\Psi_0\|_{H^N\cap W^{k,p'}}+\||x|\Psi_0\|_{L^2}+
\|\MP u_0\|_{H^N}\leq 
\varepsilon,\ \|\MP u_0\|_{W^{k,p}\cap W^{k,p'}}+
\||x|\MP u_0\|_{L^2}\leq \delta,
\end{equation*}
then the solution of \eqref{equfinal} exists on 
$[0,T]$, $T\geq c/\delta$ and 
$\|\Psi\|_{X_T}\lesssim \varepsilon$.\\
This result implies theorem \ref{thd3}.
\end{theo}

\subsection{The bootstrap argument}
\paragraph{A priori estimates} 
The aim of this paragraph and the next one is to prove that 
for $c,\varepsilon$ small enough, there exists 
$C>0$ such that for $t\leq c\delta$, 
if we have the following estimates
\begin{equation}\label{boot}
\left\{
\begin{array}{ll}
\di \|\Psi\|_{H^N}\leq C \varepsilon&\text{ (total energy)},\\
\di \||x|e^{-itH}\Psi\|_{L^2}\leq C\varepsilon,\ 
\|\Psi\|_{W^{k,p}}\leq C\delta
+\frac{C\varepsilon}{(1+t)^{d(1/2-1/p)}}
&\text{ (dispersive estimates)},
\\
\di \|\MP u\|_{W^{k,q}}+\||x|\MP u_0\|_{L^2}\leq C\delta& \text{ (transport energy)},
\end{array}\right.
\end{equation}
then the same estimates hold with
$C/2$ instead of $C$.
\begin{rmq}
We point out that the bootstrap argument is slightly different from 
the one for theorem \ref{thd5}. Indeed in large dimension, we can propagate the 
a priori bounds (up to multiplicative constants independent
of $\varepsilon,\delta$) 
on a time $c/\delta$ while for $d=3,4$ the proof implies 
$c=O(\varepsilon)$. In other words, 
if $\|\Psi_0\|_{H^N}\leq \varepsilon'<\varepsilon$ it is not clear if  
$\|\Psi(t)\|_{H^N}\lesssim \varepsilon'$ on $[0,c/\delta]$ with $\delta$ 
independent of $\varepsilon'$, see remark \ref{smallbizarre} for technical 
details.
\end{rmq}

The dispersive estimates are significantly more difficult than for $d\geq 5$ 
and are detailed in paragraph \ref{estdisp}.
\paragraph{The energy estimate} This is the same 
argument as for $d\geq 5$, from proposition \ref{totenergy} and using $3(1/2-1/p)>1$ (integrability of the decay)
\begin{equation*}
\|\Psi\|_{H^N}+\|\MP u\|_{H^N}\lesssim 
C_1\varepsilon \text{exp}\bigg(C_2C
\bigg(2\delta t+\frac{\varepsilon}{3(1/2-1/p)-1}\bigg)\bigg),
\end{equation*}
so that for $C$ large enough, $\varepsilon,c$ small enough, $t\leq c/\delta$
\begin{equation}\label{energyd3}
\|\Psi\|_{H^N}+\|\MP u\|_{H^N}\leq C\varepsilon/2.
\end{equation}
\paragraph{The transport energy estimate}
The $W^{k,q}$ estimate is a consequence of proposition
\ref{energytransport} as for $d\geq 5$: for $k$ even
large enough
\begin{eqnarray}\nonumber
\frac{d}{dt}\|\MP u\|_{W^{k,p}}\leq
C_1(\|\MP u\|_{W^{k,p}}+\|\MQ u\|_{W^{k,p}})\|\MP u\|_{W^{k,p}}
\lesssim C_1C^2\bigg(\delta+\frac{\varepsilon}{t^{1+}}\bigg)
\delta\\
\Rightarrow 
\|\MP u\|_{W^{k,p}}\leq C_1C^2\delta(c+C_2\varepsilon)
\leq \frac{C}{2}\delta, \label{transportenergyd3easy}
\end{eqnarray}
for $c,\varepsilon$ small enough.
The same estimate (with indices $q\leq p$) applied 
again gives 
\begin{equation*}
\|\MP u\|_{W^{k,q}}\leq C_1\int_0^t(\|\MP u\|_{W^{k,p}}
+\|\MQ u\|_{W^{k,p}})\|\MP u\|_{W^{k,q}}ds
\leq \frac{C\delta}{2}.
\end{equation*}
For the weighted estimate we follow a similar energy method. 
First multiply the equation 
on $\MP u$ by $x_j$ :
\begin{eqnarray*}
\partial_t (x_j\MP u)+x_j\MP (u\cdot \n \MP u+
\MP u\cdot \n \MQ u)&=&
\partial_t (x_j\MP u)+\MP \big(u\cdot \n (x_j\MP u)
+x_j\MP u\cdot \n \MQ u\big)
\\
&&+[x_j,\MP] (u\cdot \n\MP u+\MP u\cdot \n \MQ u)
+\MP \big([x_j,u\cdot \n]\MP u\big)\\
&=&=0.
\end{eqnarray*}
The operator $[x_j,\MP]$ is the Fourier multiplier 
of symbol $i\partial_{\xi_j}\MP(\xi)$ which is 
dominated by $1/|\xi|$ therefore it is bounded
$\dot{H}^{-1}\mapsto L^2$. From the embedding 
$\dot{H}^1\subset L^6$, $[x_j,\MP]$ is bounded 
$L^{6/5}\to L^2$. We deduce the following bound 
\begin{eqnarray*}
\int_{\R^d} [x_j,\MP] (u\cdot \n\MP u+\MP u\cdot \n \MQ u)
\cdot (x_j\MP u)dx&\lesssim& \|u\cdot \n\MP u+\MP u\cdot 
\n \MQ u\|_{L^{6/5}}\|x_j\MP u\|_{L^2}\\
&\lesssim& \|\MP u\|_{W^{k,6/5}}\big(\|\MP u\|_{W^{k,p}}
+\|\MQ u\|_{W^{k,p}}\big)\|x_j\MP u\|_{L^2}.
\end{eqnarray*}
Using an integration by parts
\begin{eqnarray*}
\int_{\R^d}\MP \big(u\cdot \n (x_j\MP u)\big)\cdot 
x_j\MP udx&=&
\int_{\R^d}\big(u\cdot \n (x_j\MP u)\big)\cdot 
([\MP, x_j]\MP u+x_j\MP u)dx\\
&=&\int_{\R^d}
-\frac{\text{div} \MQ u}{2}\big(|x_j\MP u|^2+
(x_j\MP u)[\MP,x_j]\MP u\big)\\
&&\hspace{3cm}-u\cdot \n ([x_j,\MP]\MP u)\cdot x_j\MP udx
\\
&\lesssim& \|\MQ u\|_{W^{k,p}}(\|x_j\MP u\|_{L^2}+
\|\MP u\|_{L^{6/5}})\|x_j\MP u\|_{L^2}\\
&& +(\|\MP u\|_{W^{k,p}}+\|\MQ u\|_{W^{k,p}})
\|\MP u\|_{L^2}\|x_j\MP u\|_{L^2}.
\end{eqnarray*}
Similarly 
\begin{eqnarray*}
\int_{\R^d}x_j\MP u\cdot \n \MQ u\cdot (x_j\MP u)dx
&\leq& \|\MQ u\|_{W^{k,p}}\|x_j\MP u\|_{L^2}^2,\\
\int_{\R^d}\MP ([x_j,u\cdot \n]\MP u)\cdot x_j\MP udx&=&
-\int_{\R^d}\MP (u_j\MP u)\cdot x_j\MP udx\\
&\lesssim &
\big(\|\MP u\|_{W^{k,p}}+\|\MQ u\|_{W^{k,p}}\big)
\|\MP u\|_{L^2}\|x_j\MP u\|_{L^2}.
\end{eqnarray*}
From these estimates we deduce 
\begin{equation*}
\frac{d}{dt}\|x_j\MP u\|_{L^2}^2\leq C 
\|x_j\MP u\|_{L^2}\big(\|\MQ u\|_{W^{k,p}}
+\|\MP u\|_{W^{k,p}\cap W^{k,q}}+\|x_j\MP u\|_{L^2}\big)^2,
\end{equation*}
which readily yields by integration in time and 
the bootstrap assumption \eqref{boot}
\begin{equation}\label{transportenergyd3}
\|x_j\MP u\|_{L^2}\leq C_1C^2(\delta t+\varepsilon)\delta\leq \frac{C\delta}{2}.
\end{equation}
\subsection{The dispersive estimates}\label{estdisp} We start from 
\eqref{equfinal} that reads $\partial_t\Psi=iH\psi
+ \mathcal{D}(\Psi)+\mathcal{T}(\Psi,\MP u)+R$,
with $\mathcal{D}$ the first two lines of nonlinear 
terms (quadratic dispersive terms), $\mathcal{T}$ the third line (dispersive-transport, and transport-transport) 
and $R$ cubic. Equivalently 
\begin{equation*}
\Psi (t)=e^{itH}\Psi_0+\int_0^t e^{i(t-s)H}\big(
\mathcal{D}(\Psi)+\mathcal{T}(\Psi,\MP u)+R\big)(s)ds
\end{equation*}
The linear part is not difficult to control: 
\begin{equation}\label{linweight}
\||x|e^{-itH}e^{itH}\Psi_0\|_{L^2}=\||x|\Psi_0\|_{L^2},
\end{equation}
\begin{equation}\label{lindecay}
\|e^{itH}\Psi_0\|_{W^{k,p}}\lesssim
\frac{\|\Psi_0\|_{H^N\cap W^{k,p'}}}{(1+t)^{3(1/2-1/p)}}.
\end{equation}
The terms in $\mathcal{D}$ and $\mathcal{T}$ are not estimated exactly similarly. Basically the control of $\mathcal{D}$ is quite difficult, but amounts to 
a straightforward modification
of the estimates in \cite{AudHasp2}, while 
$\mathcal{T}$ is new but 
a bit easier to control. For completeness, the key arguments 
to estimate $\mathcal{D}$ are provided in the appendix \ref{spacetime}.\\
The nonlinearity $\mathcal{T}$ contains four terms that are all very
similar. For conciseness we only detail how to estimate 
$U^{-1}\n (\MP u\cdot w_1)$, which contains all the difficulties of the 
other terms plus a singular factor $U^{-1}$.
Finally, $R$ contains cubic terms easier to control.
To fix ideas, we estimate the term 
$U^{-1}\n B[\MQ(u\cdot \n \MP u),\MQ u]$ that appears in the 
proof of lemma \ref{formenorm}).

\paragraph{Weighted bounds}
\subparagraph{Quadratic term}
We shall detail the estimate of $xe^{-itH}\int_0^te^{i(t-s)H}U^{-1}\n(\MP u\cdot\Psi)ds$.
Since $w_1=U(\Psi-\overline{\Psi})/2$, we have 
$$U^{-1}\n (\MP u\cdot w_1)
=U^{-1}\n \bigg(\MP u\cdot U\frac{\Psi-\overline{\Psi}}{2}\bigg),$$
so we define $\varphi=U\Psi$ and consider the term $U^{-1}\n (\MP u\cdot \varphi)$.
The weighted estimate amounts to control
\begin{equation*}
\bigg\|\n_\xi\bigg(\int_{[0,t]\times\R^d}i\xi  e^{-isH(\xi}
 U^{-1}(\xi)\widehat{\MP u}(\xi-\eta)
\cdot  \widehat{\varphi}(\eta)d\eta\bigg)ds\bigg\|_2,
\end{equation*}
therefore setting $m(\xi,s)=i\partial_{\xi_j}
(\xi U^{-1}(\xi)e^{-isH(\xi)})$
\begin{eqnarray*}
x_je^{-itH}\int_0^tU^{-1}\n (\MP u\cdot \varphi)ds&=&\mathcal{F}^{-1}\bigg(
\int_0^t\int_{\R^d} m(\xi,s)
 \widehat{\MP u}(\xi-\eta)
\cdot  \widehat{\varphi}(\eta)d\eta ds\bigg)\\
&&+\int_{0}^t e^{-isH}U^{-1}\n 
(x_j \MP u\cdot  \varphi\big)ds.
\end{eqnarray*}
We have $m=i\partial_{\xi_j}(\xi U^{-1})e^{-isH}
+\xi U^{-1}e^{-isH}s\partial_{\xi_j}H=m_1+m_2$. 
From elementary computations $m_1=m_3\sqrt{1+|\xi|^2}/|\xi|$ 
with $m_3$ a bounded multiplier, 
therefore it is continuous $W^{1,6/5}\to L^2$ and from Minkowski's inequality
\begin{equation}\label{weight1}
\bigg\|\int_0^tm_1(D)(\MP u \cdot \varphi)ds\bigg\|_{L^2}\lesssim
\int_0^t\|\MP u\cdot \varphi\|_{W^{1,6/5}}ds\lesssim 
\int_0^t\|\MP u\|_{W^{1,6/5}}\|\varphi\|_{W^{k,p}}ds,
\end{equation}
similarly $m_2\lesssim (1+|\xi|)^2s$ so 
\begin{equation}\label{weight2}
\bigg\|\int_0^tm_2(D)(\MP u \cdot \varphi)ds\bigg\|_{L^2}
\lesssim \int_0^t\|\MP u\cdot \varphi\|_{H^2}ds\lesssim 
\int_0^ts\|\MP u\|_{H^2}\|\varphi\|_{W^{k,p}}ds.
\end{equation}
Next we use a frequency truncation $\chi(D)$, with $\chi\in C_c^\infty,\ 
\chi\equiv 1$ near $0$, and split 
\begin{eqnarray*}
\int_{0}^t e^{-isH}U^{-1}\n  (x_j \MP u\cdot  \varphi\big)ds
=\int_{0}^t e^{-isH}(\chi +1-\chi)U^{-1}\n  (x_j \MP u\cdot  \varphi\big)ds.
\end{eqnarray*}
The low frequency part is estimated using the boundedness of $\chi U^{-1}\n:\
L^{6/5}\to L^2$
\begin{equation}\label{weight3}
\bigg\|\int_{0}^t e^{-isH}\chi U^{-1}\n  (x_j \MP u\cdot  \varphi\big)ds\bigg\|_2
\lesssim \int_0^t\|x_j\MP u\cdot \varphi\|_{6/5}ds\lesssim \int_0^t\|x_j\MP u\|_2 
\|\varphi\|_{3}ds.
\end{equation}
For the high frequency part, we use that $(1-\chi)U^{-1}$ is a bounded multiplier,
the identity 
$$\n (x_j\MP u\cdot \varphi)=(\MP u\cdot \varphi)e_j
+(\n \varphi)\cdot (x_j\MP u)+\n(\MP u)\cdot \big([x_j,Ue^{isH}]e^{-isH}\Psi
+Ue^{isH}(x_je^{-isH}\Psi)\big),$$
and the bound  $|[x_j,Ue^{isH}]|(\xi)\lesssim (1+s)(1+|\xi|)$, so
\begin{eqnarray}
\nonumber
\bigg\|\int_{0}^t e^{-isH}(1-\chi)U^{-1}\n  (x_j \MP u\cdot  \varphi\big)ds\bigg\|_2
&\lesssim&\int_0^t (\|\MP u\|_2+\|x_j\MP u\|_2)\|\varphi\|_{W^{k,p}}
\\
\nonumber
&&\hspace{1cm} +\|\MP u\|_{H^1}(1+s)\|\Psi\|_{W^{k,p}}
 \\ \label{weight4}
&&\hspace{15mm} +\|\MP u\|_{W^{k,p}}\|x_je^{-isH}\Psi\|_2ds.
\end{eqnarray}
From estimates \eqref{weight1},\eqref{weight2},\eqref{weight3},\eqref{weight4}
and the bootstrap assumptions \eqref{boot} we get for $c,\varepsilon$ 
small enough, $t\leq c/\delta$
\begin{eqnarray}
\nonumber
\bigg\|x_je^{-itH}\int_0^te^{i(t-s)H}U^{-1}\n (\MP u \cdot U\Psi)ds\bigg\|_2&\leq&
C^2C_1\int_0^t (1+s)\delta\bigg(\delta+\frac{\varepsilon}{(1+s)^{3(1/2-1/p)}}\bigg)
+\delta \varepsilon ds\\
\label{quadweight}
&\leq& C^2C_1(c^2+c\varepsilon).
\end{eqnarray}
\begin{rmq}\label{smallbizarre}
The weighted estimate is the only point in the proof where we need 
$c\lesssim \varepsilon$. More precisely, it is due to the commutator term 
$[x_j,e^{-isH}]=s(\n H)$ which causes a strong loss of decay in the estimate 
\eqref{weight2}.
\end{rmq}
\subparagraph{Cubic term}
From similar computations, we end up estimating 
terms like
\begin{eqnarray}
\label{cubic1}
\int_0^te^{-isH}U^{-1}\n B[\MQ (u\cdot 
\n \MP u), e^{isH}x_je^{-isH}\MQ u ]d\eta,\\
\int_0^t e^{-isH}s(\partial_j H)U^{-1}\n B[\MQ(u\cdot \n \MP u),
\MQ u]ds... \label{cubic2}
\end{eqnarray}
For the first one, since the symbol of $B$ 
is $(a'(1)-1)/(2(2+|\eta|^2+
|\xi-\eta|^2))$ ) we may use the boundedness 
of the bilinear multiplier $\n  B$ 
\begin{eqnarray}
\nonumber
\bigg\|
\int_0^te^{-isH}U^{-1}\n B[\MQ (u\cdot 
\n \MP u), e^{isH}x_je^{-isH}\MQ u ]d\eta\bigg\|_2
&\lesssim& \int_0^t\|u\cdot \n \MP u\|_\infty\|x_j e^{-isH}
\MQ u\|_2ds\\
\nonumber 
&\lesssim& \int_0^t\|u\|_{W^{k,p}}\|\MP u\|_{W^{k,p}} 
\|x_je^{-isH}\MQ u\|_2ds\\
\nonumber 
&\lesssim& \int_0^tC^3\bigg(\delta+
\frac{\varepsilon}{(1+s)^{3(1/2-1/p)}}\bigg)^2\varepsilon ds
\\
\label{cubic3}
&\lesssim& C^3(\delta^2t+\varepsilon^2)\varepsilon,
\end{eqnarray}
similarly for \eqref{cubic2} 
\begin{eqnarray}
\nonumber
\bigg\|\int_0^t e^{-isH}s(\partial_j H)U^{-1}\n B[\MQ(u\cdot \n \MP u),
\MQ u]ds\bigg\|_2&\lesssim&\int_0^t s\|u\|_{W^{k,p}}
\|\MQ u\|_{W^{k,p}}\|\MP u\|_{W^{k,p}}ds\\
\nonumber
&\lesssim& \int_0^t C^3s\bigg(\delta+\frac{\varepsilon}
{(1+s)^{3(1/2-1/p)}}\bigg)^3ds\\
\label{cubic4}
&\lesssim& C^3(t^2\delta^3+\varepsilon^3).
\end{eqnarray}
\eqref{cubic3} and \eqref{cubic4} are clearly more than 
enough to close the weighted estimate.
\subparagraph{Closing the bound}
The estimates \eqref{linweight} (linear), quadratic 
\eqref{quadweight} (quadratic, see also the appendix 
for the purely dispersive terms)  and \eqref{cubic3}
,\eqref{cubic4} (cubic) lead to 
\begin{equation}\label{yesweight}
 \|xe^{-itH}\Psi(t)\|_{L^2}\leq C_1\varepsilon+C^2C_1(c^2+c\varepsilon),
\end{equation}
which gives the first part of the dispersive estimate by choosing $C$ large enough, $c=\varepsilon$ small 
enough.
\paragraph{Bounds in $W^{k,p}$} 
The computations are done ``up to choosing $k$, $N$ 
larger''.
\subparagraph{Quadratic term}
As previously 
we focus on $U^{-1}\n (\MP u\cdot U\Psi):=U^{-1}\n (\MP u\cdot  \varphi)$.
We can assume $t\geq 2$, indeed for $t\leq 2$ 
by Sobolev's embedding and for $N$ large enough
\begin{eqnarray*}
\|\int_0^te^{i(-s)H}U^{-1}\n (\MP u\cdot \varphi)ds\|
_{W^{k,p}}&\lesssim& \int_0^t\|\n (\MP u\cdot 
\varphi)\|_{H^{N-1}}ds
 \\
&\lesssim &\|\MP u\|_{L^\infty([0,2],H^N)}
\|\varphi\|_{L^\infty([0,2],H^N)}
\leq C\varepsilon^2.
\end{eqnarray*} 
Since $p>6$, $3(1/2-1/p):=1+\gamma>1$. Minkowski's inequality and the dispersive estimate imply
\begin{eqnarray*}
\bigg\|\int_0^te^{i(t-s)H}U^{-1}\n (\MP u\cdot \varphi)ds\bigg\|_{W^{k,p}}
&\lesssim& \int_0^{t-1}\frac{\|\MP u\cdot \varphi\|_{W^{k+1,p'}}}{(t-s)^{1+\gamma}}ds
+\int_{t-1}^t\|\MP u\cdot \varphi\|_{H^{k+2}}ds\\
&\lesssim& \int_0^{t-1}\frac{\|\MP u\|_{W^{k,p'}}\| \varphi\|_{H^N}+
\|\MP u\|_{H^{k+1}}\|\varphi\|_{W^{k,q}}}{(t-s)^{1+\gamma}}ds\\
&& +\int_{t-1}^t\|\MP u\|_{W^{k,p}}\|\varphi\|_{H^N}+\|\MP u\|_{H^{N}}
\|\varphi\|_{W^{k,p}}ds,
\end{eqnarray*}
with $1/q+1/2=1/p'$. By interpolation, the bootstrap 
assumption gives with $\theta/p+(1-\theta)/2=1/q$
\begin{eqnarray*}
\|\MP u\|_{H^{k+1}}\leq C\varepsilon^{1/(N-k)}\delta^
{1-1/(N-k)},\  \|\varphi\|_{W^{k,q}}&\leq &
\|\varphi\|_{W^{k,p}}^\theta\|\varphi\|_{H^k}^{1-\theta}\\
&\leq& C\bigg(\delta+\frac{\varepsilon}{(1+s)^{1+\gamma}}
\bigg)^\theta \varepsilon^{1-\theta}.
\end{eqnarray*}
Note that $\theta=1/(p/2-1)$ with $p>6$ thus $\theta<1/2$. 
We choose $6<p<8$, so that $\theta>1/3$, from elementary computations, 
and for $t\geq 2$
\begin{eqnarray}
\nonumber
\bigg\|\int_0^te^{i(t-s)H}U^{-1}\n (\MP u\cdot \varphi)ds\bigg\|_{W^{k,p}}
&\leq& C_1 C^2\int_0^{t-1}\frac{\delta\varepsilon}{(t-s)^{1+\gamma}}+
\frac{\delta^{1-\frac{1}{N-k}+\theta}\varepsilon^{1-\theta+\frac{1}{N-k}}}
{(t-s)^{1+\gamma}}
\\
\nonumber
&&\hspace{3cm}+\frac{\delta^{1-\frac{1}{N-k}}\varepsilon^{1+\frac{1}{N-k}}}
{(t-s)^{1+\gamma}(1+s)^{\theta(1+\gamma)}}ds\\
\nonumber
&& +C_1C^2\bigg(\delta \varepsilon+\frac{\varepsilon^2}{(1+t)^{1+\gamma}}\bigg)
\\
&\leq &C_1C^2\bigg(\delta\varepsilon+\frac{\varepsilon^2}{(1+t^2)^{1+\gamma}}
+\frac{\delta^{1-\frac{1}{N-k}}\varepsilon^{1+\frac{1}{N-k}}}
{t^{\theta(1+\gamma)}}\bigg)\nonumber\\
\label{quaddecay}
&\leq& C_1C^2\bigg(\delta\varepsilon+\frac{\varepsilon^2}{(1+t^2)^{1+\gamma}}
+\delta\varepsilon
+\frac{\varepsilon^2}{t^{\frac{N-k}{3}}}\bigg).
\end{eqnarray}
\subparagraph{Cubic term} As for the quadratic terms,
we split the integral on $[0,t-1]\cup [t-1,t]$. The integral on $[t-1,t]$ is easily controlled, for the other 
part, using again the boundedness of $\n B$, and choosing 
$1/q=1/p'-1/2$
\begin{eqnarray*}
\bigg\|\int_0^{t-1}e^{i(t-s)H}U^{-1}\n B[\MQ (u\cdot \n \MP u),
\MQ u]ds\bigg\|_{W^{k,p}}
&\lesssim& \int_0^{t-1}\frac{\|u\cdot \n \MP u\|_{W^{k,q}}\|\MQ u\|_{H^N}}
{(t-s)^{1+\gamma}}ds\\
&\lesssim& \int_0^{t-1}\frac{\|u\|_{W^{k,p}}\|\MP u\|_{H^N}}{(t-s)^{1+\gamma}}
\|\MQ u\|_{H^N}ds\\
&\lesssim&C^3\varepsilon^2\bigg(\delta+
\frac{\varepsilon}{(1+t)^{1+\gamma}}\bigg).
\end{eqnarray*}
\subparagraph{Closing the bound}
From \eqref{lindecay}, \eqref{quaddecay} and the cubic estimates above we deduce 
\begin{equation}\label{yesdecay}
\|\Psi(t)\|_{W^{k,p}}\leq \frac{C_1\varepsilon}{(1+t)^{1+\gamma}}
+C^2C_1\varepsilon\bigg(\delta+\frac{\varepsilon}{(1+t)^{1+\gamma}}\bigg)
\leq \frac{C}{2}\bigg(\delta+\frac{\varepsilon}{(1+t)^{1+\gamma}}\bigg).
\end{equation}
\paragraph{End of proof}
As for theorem \ref{thd5}, we close the bootstrap 
argument thanks to the energy estimate \eqref{energyd3}, the transport energy estimates 
\eqref{transportenergyd3easy},\eqref{transportenergyd3},
and the dispersive estimates \eqref{yesweight},
\eqref{yesdecay}.

\section{An example of blow up}\label{secblowup}
We consider in this section the special case 
of quantum fluid, where $K$ is proportional to 
$1/\rho$. More precisely, if $\psi$ is a 
smooth solution of 
\begin{equation}\label{GP}
i\partial_t\psi+\Delta \psi=\frac{g(|\psi|^2)\psi}{2},
\end{equation}
that does not cancel, the so-called Madelung transform 
$\psi=\sqrt{\rho}e^{i\phi}$, $u=\n \phi$ is well
defined and $(\rho,u)$ satisfy 
\begin{equation}\label{fluidQ}
\left\{
\begin{array}{lll}
 \partial_t\rho +\text{div}(\rho u)&=&0,\\
 \partial_tu+u\cdot \n u+\n g(\rho)&=&2\n\bigg(
 \frac{\Delta \sqrt{\rho}}{\sqrt{\rho}}\bigg).
\end{array}
\right.
\end{equation}
As pointed out in the review article 
\cite{CDS}, the Madelung transform is a major tool
to study nonlinear Schr\"odinger equations with 
non zero boundary conditions at infinity, with 
a (technical but important) drawback that it 
becomes singular in presence of vacuum, that is 
when $\rho$ vanishes. Cancellation of $\psi$ is often
labelled as vortex formation in the framework of 
NLS. We construct here in dimension one an example 
of solution such that vacuum appears in finite time.
\begin{prop}
Let $\psi_0$ real valued such that 
$$
1-\psi_0\in \mathcal{S}(\R^d),\ \psi_0>0\text{ on }\R^d\setminus \{0\},\ \psi_0(0)=0,\ \Delta\psi_0(0)
\neq 0.
$$
Then there exists a local solution to \eqref{GP} 
with $\psi|_{t=0}=\psi_0$, and $T>0$ such that 
$\psi(x,t)>0$ on $]0,T]\times \R^d$.\\
Consequently, there exists a solution to 
\eqref{fluidQ} that blows up in finite time.
\end{prop}
\begin{proof}
Since $1-\psi_0$ is smooth, the existence of 
a smooth solution to \eqref{GP} is a consequence of 
the standard theory for NLS equations. From direct 
computations 
\begin{eqnarray}
\partial_t|\psi|^2&=&-2\text{Im}(\overline{\psi}
\Delta \psi),\\
\partial_t^2|\psi|^2&=&2|\Delta \psi|^2-\text{Re}\big(
g\overline{\psi}\Delta \psi+\overline{\psi}
\Delta^2\psi-\overline{\psi}\Delta(g\psi)
\big).
\end{eqnarray}
Since $\psi_0$ is real valued, we deduce
\begin{equation}
\forall\,x\in \R^d,\ \partial_t|\psi(x,0)|^2=0,\\
\partial_t^2|\psi|^2(0,0)=2|\Delta \psi_0(0)|^2>0.
\end{equation}
By continuity, there exists $\alpha>0$ such that 
$\partial_t^2|\psi(x,t)|^2\geq \alpha$ on a neighbourhood $U$ of 
$(x,t)=(0,0)$, we deduce by Taylor expansion 
\begin{equation*}
\forall\,(x,t)\in U,\ 
|\psi(x,t)|^2=|\psi(x,0)|^2+\int_0^t(t-s)\partial_t^2|\psi(x,s)|^2ds\geq \alpha t^2/2.
\end{equation*}
Now by continuity, for $(x,t)\in U^c$, $t$ small enough, $\psi(x,t)$ does not vanish 
hence for $t$ small enough, $\psi(\cdot,t)>0$ on $\R^d$. Thanks to the reversibility\footnote{The map $\psi(t)\to :
\overline{\psi}(-t)$ leaves the solution set invariant,
or equivalently $(\rho,u)(t)\to (\rho,-u)(-t)$}
of the equations, starting with initial data $\psi(\cdot,t)$ and going backwards 
in time provides a solution of \eqref{GP} that cancels at $x=0$ in a finite time
$T^*$. 
The (inverse)Madelung transform $\psi\to (\rho,u)=\big(|\psi|^2,\ 
\text{Im}(\frac{\overline{\psi}\n \psi}{|\psi|^2})\big)$ then gives a smooth solution of 
\eqref{fluidQ} initially without vacuum, but with formation of vacuum at 
$x=0, t=T^*$. This implies blow up of $u$ according to 
the method of characteristics: define $X(t)$ as the flow 
associated to $u$, $X'(t)=u(t,X(t))$, we have 
\begin{equation*}
\frac{d}{dt}\rho(t,X(t))=-\rho \text{div}u,
\end{equation*}
hence $\rho(t,X(t))=\rho_0(X(0))e^{-\int_0^t\text{div} uds}$, 
the cancellation of 
$\rho$ implies $\|u\|_{L^1_{T^*}W^{1,\infty}}=\infty$.
\end{proof}
\begin{rmq}
 The blow up is not linked to vorticity, indeed the 
 initial data $\psi_0$ is real positive, thus its 
 index is zero.
\end{rmq}

\appendix  
\section{The total energy estimate}\label{proofenergy}
This section is devoted to the proof of proposition 
\ref{totenergy}. 
This is essentially a variation on the estimates in \cite{Benzoni1}, that we include here 
for self-containedness.\\
We define $z=u+iw$ so that according to \eqref{EEK}, $(\rho,z)$ satisfy
\begin{equation}
\left\{
\begin{array}{lll}
\partial_t\rho+\text{div}(\rho u)&=&0,\\
\partial_tz+u\cdot \n z+i\n z\cdot w+i\n (a\text{div}z)+\n(g(\rho))&=&0.
\end{array}\right.
\end{equation}
A direct energy method where one takes the scalar product of the 
second equation with $z$ and integrate causes loss of derivatives 
due to the term $i\n z\cdot w$. The remedy is 
done in two times : first use a gauge $\varphi_n(\rho)$
and derive an energy bound 
$\frac{d}{dt}\int |\MQ(\varphi_n\Delta^nz)|^2dx$ for $n\in \N$, this estimate contains a loss of derivatives,
but an other gauge estimate on $\MP (\phi_n\Delta^nz)$
for an appropriate choice of $\phi_n$ compensates 
exactly the loss.
\\
In what follows, $R$ stands for a nonlinear term (quadratic of higher)
that contains only derivatives of $z$ of order at 
most $2n$, and is thus without loss of derivatives, $I_R$ is an integrated term
which is dominated by $\|(r,u)\|_{\mathcal{W}^{1,\infty}}
\|(r,u)\|_{\mathcal{H}^{2n}}^2$.\\
We will need the following lemma :
\begin{lemma}[\cite{BDD}, lemma 3.1]\label{magiecurl}
For $Z\in C^1(\R^d,\C),\ W\in C^1(\R^d,\R^d)$, with limit $0$ at infinity,
\begin{equation*}
 2i\textrm{Im}\int_{\R^d}Z^*\cdot\n_0 Z\cdot Wdx=\int Z^*
 \cdot\textrm{curl}W\cdot Zdx.
\end{equation*}
In particular, if $W$ is a gradient, the integral is $0$.
\end{lemma}

\paragraph{Equation on $\varphi_n(\rho)\Delta^nz$}
We recall $a=\sqrt{\rho K}$, $w=\sqrt{K/\rho}\n \rho$, and start from
\begin{equation*}
\partial_tz+u\cdot \nabla z+i\nabla z\cdot w
+i\nabla(a\text{div}z)+g'w=0.
\end{equation*}
Apply  $\varphi_n\Delta^n$ together with  the 
commutator identity
$$\n(\Delta^n(a\text{div}z))=\n(a\text{div}\Delta^nz)
+\n\big((2n\nabla a)\cdot \nabla \text{div}\Delta^{n-1}z\big)+R,$$ 
and 
$\varphi_n\nabla(a\text{div}\Delta^nz)=\nabla (a\text{div}
(\varphi_n\Delta^nz))-a\nabla \varphi_n\text{div}\Delta^nz-a\nabla 
\Delta^nz\cdot \nabla\varphi_n+R$,
\begin{eqnarray*}
\partial_t(\varphi_n\Delta^nz)+u\cdot \nabla 
(\varphi_n\Delta^nz)+i\varphi_n\nabla(\Delta^nz)\cdot w
+i\nabla (a\text{div}(\varphi_n\Delta^nz))
+g'\varphi_n\Delta^n w\\
+2in\varphi_n\nabla (\nabla a\cdot \nabla \text{div}
\Delta^{n-1}z)-ia(\nabla \varphi_n)\text{div}\Delta^nz
-ia\nabla \Delta^nz\cdot \nabla \varphi_n=R,
\end{eqnarray*}
so using $\text{div}=\text{div}\circ \MQ$ and 
$\nabla \text{div}\Delta^{n-1}z
=\Delta^n \mathbb{Q}z$
\begin{eqnarray}\nonumber
 \partial_t(\varphi_n\Delta^nz)+u\cdot \nabla 
(\varphi_n\Delta^nz)
+i\nabla (a\text{div}(\varphi_n\Delta^nz))
+g'\varphi_n\Delta^n w&&\\
\label{start}
+i\varphi_n \nabla(\Delta^nz)\cdot w
+2in\varphi_n\nabla (\Delta^{n}\MQ z)\cdot \nabla a
-ia(\nabla \varphi_n)\text{div}\Delta^n\MQ z
-ia\nabla \Delta^nz\cdot \nabla \varphi_n&=&R.
\end{eqnarray}
The loss of derivative is caused by the left hand side 
of the second line. 
For $\varphi_n=a^n \sqrt{\rho}$, 
and denoting $\n_0:=\n -I_d\text{div}$, we find
\begin{eqnarray}\nonumber
a^n\sqrt{\rho}\nabla (\Delta^nz)\cdot w+2na^n\sqrt{\rho}
\nabla (\Delta^n\MQ z)\cdot \nabla a-a\nabla(a^n\sqrt{\rho})
\text{div}\Delta^n\MQ z-a\nabla \Delta^nz\cdot \nabla 
(a^n\sqrt{\rho})\\
\nonumber
=\frac{a^{n+1}}{\sqrt{\rho}}\nabla (\Delta^nz)\cdot 
\nabla \rho
+2na^na'\sqrt{\rho}
\nabla (\Delta^n\MQ z)\cdot \nabla \rho
-\bigg(na^na'\sqrt{\rho}+\frac{a^{n+1}}{2\sqrt{\rho}}\bigg)
\text{div}(\Delta^n\MQ z)\nabla \rho\\
\nonumber
-\bigg(na^na'\sqrt{\rho}+\frac{a^{n+1}}{2\sqrt{\rho}}\bigg)
\nabla \Delta^nz\cdot \nabla 
\rho
\\
\label{loss1}
=\bigg(na^n\sqrt{\rho}a'+\frac{a^{n+1}}{2\sqrt{\rho}}\bigg)
\nabla_0 \Delta^n\MQ z
\cdot\nabla \rho
+\bigg(-na^n\sqrt{\rho}a'+\frac{a^{n+1}}{2\sqrt{\rho}}\bigg)
\nabla \Delta^n\MP z\cdot \nabla \rho.
\end{eqnarray}
We write $na^n\sqrt{\rho}a'+\frac{a^{n+1}}{2\sqrt{\rho}}
=\varphi_n(na'+a/(2\rho))$, commute $\varphi_n$ with $\n_0$, 
then we use that for $n\geq 1$, $\MQ \Delta^n$ is a differential operator of order $2n$,
so $\varphi_n\MQ \Delta^n\cdot=\MQ \Delta^n(\varphi_n\cdot)+[\varphi_n,\MQ \Delta^n]\cdot
=\MQ \big(\varphi_n\Delta^n\cdot)+P\cdot $ with $P$ a differential operator of order
$2n-1$, therefore 
\begin{equation}\label{commute}
\bigg(na^n\sqrt{\rho}a'+\frac{a^{n+1}}{2\sqrt{\rho}}\bigg)
\nabla_0 \Delta^n\MQ z
\cdot\nabla \rho=\bigg(na'+\frac{a}{2\rho}\bigg)
\nabla_0 \MQ (\varphi_n\Delta^nz)
\cdot\nabla \rho+R.
\end{equation}
Plugging \eqref{loss1} and \eqref{commute} in \eqref{start} we get
\begin{eqnarray}\label{reform}
\partial_t (\varphi_n\Delta^nz)&+&u\cdot \nabla (\varphi_n\Delta^nz)
+g'\varphi_n\Delta^nw+i\nabla (a\text{div}(\varphi_n\Delta^nz))\\
&=&\nonumber
-i\bigg(na'+\frac{a}{2\rho}\bigg)
\nabla_0 \MQ (\varphi_n\Delta^nz)
\cdot\nabla \rho
-i\bigg(-na^n\sqrt{\rho}a'+\frac{a^{n+1}}{2\sqrt{\rho}}\bigg)
\nabla \Delta^n\MP z\cdot \nabla \rho+R.
\end{eqnarray}
Note that $g'\varphi_n\Delta^nw =g'(1)\varphi_n(1)\Delta^nw+R=2\Delta^nw+R$ is without loss of 
derivatives but contains a linear term that can not 
be neglected for long time dynamics. 
\paragraph{Energy estimate for $\MQ \varphi_n\Delta^nz$} Take the 
scalar product of \eqref{reform} with $\MQ \varphi_n\Delta^nz$, integrate 
in space, and use lemma \ref{magiecurl}
\begin{eqnarray}\nonumber
\frac{d}{2dt}\int|\MQ\varphi_n\Delta^nz|^2dx
+\text{Re}\int \big(u\cdot \nabla (\varphi_n\Delta^nz)
+2\Delta^nw\big)\cdot\overline{\MQ\varphi_n\Delta^nz}
dx
\\
\nonumber
=\text{Im}\int \overline{\MQ\varphi_n\Delta^nz}\cdot\bigg[
\bigg(na'+\frac{a}{2\rho}\bigg)\nabla_0 \MQ (\varphi_n\Delta^nz)
+ \bigg(-na^n\sqrt{\rho}a'+\frac{a^{n+1}}{2\sqrt{\rho}}\bigg)
\nabla \Delta^n\MP z\bigg]\cdot\nabla \rho dx+I_R
\\
=\text{Im}\int \overline{\MQ\varphi_n\Delta^nz}\cdot\bigg(-na^n\sqrt{\rho}a'
+\frac{a^{n+1}}{2\sqrt{\rho}}\bigg)\nabla \Delta^n\MP z\cdot\nabla \rho dx
+I_R,
\end{eqnarray}
where $I_R=\int R\cdot \MQ(\varphi_n\Delta^nz)dx$, 
for more details on the generic estimate $I_R$ we refer
to \cite{AudHasp2}.\\
The right hand side is an unavoidable loss of derivative, the second term 
on the left hand side rewrites with the convention of summation on repeated 
indices, and using $\partial_j(\MQ v)_i=\partial_i(\MQ v)_j$
\begin{eqnarray*}
 \int u_j\partial_j(\varphi_n\Delta^nz_i)\overline{(\MQ\varphi_n\Delta^nz)}_i)dx&=&
-\int \text{div}(u)\varphi_n\Delta^nz\cdot \overline{\MQ \varphi_n\Delta^nz}
+u_j\varphi_n\Delta^nz_i\partial_j(\overline{\MQ\varphi_n\Delta^nz})_idx\\
&=&-\int \text{div}(u)\varphi_n\Delta^nz\cdot \overline{\MQ \varphi_n\Delta^nz}
+\frac{\text{div}u}{2}|\MQ\varphi_n\Delta^nz|^2
\\
&&\hspace{35mm}+u_j(\MP\varphi_n\Delta^nz)_i
\partial_i\overline{(\MQ\varphi_n\Delta^nz)}_jdx\\
&=&-\int \text{div}(u)\varphi_n\Delta^nz\cdot \overline{\MQ \varphi_n\Delta^nz}
+\frac{\text{div}u}{2}|\MQ\varphi_n\Delta^nz|^2
\\
&&\hspace{35mm}+(\MP\varphi_n\Delta^nz)\cdot \nabla u\cdot \overline{\MQ\varphi_n\Delta^nz}dx\\
&=I_R.
\end{eqnarray*}
To summarize 
\begin{eqnarray}
\frac{d}{2dt}\int|\MQ\varphi_n\Delta^nz|^2dx
&+&\text{Re}\int 2\Delta^nw
\cdot\overline{\MQ\varphi_n\Delta^nz}dx\\
\label{lossQ}
&=&\text{Im}\int \overline{\MQ\varphi_n\Delta^nz}\cdot\bigg(-na^n\sqrt{\rho}a'
+\frac{a^{n+1}}{2\sqrt{\rho}}\bigg)\nabla \Delta^n\MP z\cdot\nabla \rho dx
+I_R.
\end{eqnarray}
\paragraph{Energy estimate for $\MP (\phi_n\Delta^nz)$ and compensated 
loss} Let $\phi_n(\rho)$ be a second gauge. 
Following the same computations that led to \eqref{start}
\begin{eqnarray*}
 \partial_t(\phi_n\Delta^nz)+u\cdot \nabla 
(\phi_n\Delta^nz)+i\nabla(\phi_n\Delta^nz)\cdot w
+i\nabla (a\text{div}(\phi_n\Delta^nz))
+2g'\phi_n\Delta^n w\\
+2in\phi_n\nabla (\Delta^{n}\MQ z)\cdot \nabla a
-ia(\nabla \phi_n)\text{div}\Delta^n\MQ z
-ia\nabla \Delta^nz\cdot \nabla \phi_n=R.
\end{eqnarray*}
We take the scalar product with $\MP \phi_n\Delta^nz$ and integrate in space, the first two terms are
\begin{eqnarray*}
\text{Re}\int \big(\partial_t(\phi_n\Delta^nz)&+&u\cdot 
\nabla (\phi_n\Delta^nz)\big)\cdot 
\overline{(\MP \phi_n\Delta^nz)}
\\
&=&\frac{1}{2}\frac{d}{dt}
\int |\MP\phi_n\Delta^nz|^2dx
-\text{Re}\int \frac{\text{div}u}{2}
|\MP\phi_n\Delta^nz|^2-\MQ\phi_n\Delta^nz\cdot \nabla u\cdot
\MP\phi_n\Delta^nz\,dx\\
&=&\frac{1}{2}\frac{d}{dt}
\int |\MP\phi_n\Delta^nz|^2dx+I_R.
\end{eqnarray*}
Most of the other terms are actually neglectible
\begin{equation*}
\int \overline{(\MP\phi_n\Delta^nz)}_i
\partial_i(\phi_n\Delta^nz_j) w_jdx=-\int 
\overline{(\MP\varphi_n\Delta^nz)}_i\varphi_n
\Delta^nz_j \partial_iw_jdx=I_R,
\end{equation*}
and  from the same computation
$\int \overline{\MP\phi_n\Delta^nz}\cdot 
(2in\phi_n\nabla (\Delta^{n}\MQ z)
\cdot \nabla a-ia\nabla \Delta^nz\cdot \nabla \phi_n)
dx=I_R$. \\
We are only left with 
\begin{equation*}
\text{Im}\int \text{div}(\Delta^n\MQ z)a\nabla \phi_n
\cdot \overline{\MP\phi_n\Delta^nz}dx=
-\text{Im}\int a(\MQ \Delta^nz)\cdot 
\nabla\overline{(\MP\phi_n\Delta^nz)}\cdot \nabla \phi_n dx+I_R,
\end{equation*}
therefore 
\begin{equation}\label{lossP}
\frac{1}{2}\frac{d}{dt}\int |\MP\phi_n\Delta^nz|^2dx
=\int a(\MQ \Delta^nz)
\nabla\overline{(\MP\Delta^nz)}\cdot (\phi_n\nabla \phi_n) dx+I_R.
\end{equation}
Sum \eqref{lossQ} and \eqref{lossP}
\begin{eqnarray}\nonumber
\frac{1}{2}\frac{d}{dt}\int |\MQ\varphi_n\Delta^nz|^2+ 
|\MP\phi_n\Delta^nz|^2dx
&+&\text{Re}\int 2\Delta^nw
\cdot\overline{\MQ\varphi_n\Delta^nz}dx\\
\nonumber
&=& \text{Im}\int \varphi_n\bigg(-na^n\sqrt{\rho}a'
+\frac{a^{n+1}}{2\sqrt{\rho}}\bigg)
\overline{\MQ\Delta^nz}\cdot
\nabla \MP \Delta^nz\cdot\nabla \rho dx
\\
\label{noloss}
&&+\text{Im}\int a\phi_n\phi_n'(\MQ \Delta^nz)\cdot 
\nabla\overline{(\MP \Delta^nz)}\cdot 
\nabla \rho dx+I_R.
\end{eqnarray}
It is now apparent that the right choice 
for $\phi_n$ is a function such that
\begin{equation*}
a\phi_n\phi_n'=\varphi_n\bigg(-na^n\sqrt{\rho}a'+
\frac{a^{n+1}}{2\sqrt{\rho}}\bigg)\Leftrightarrow 
(\phi_n^2)'=-2na^{2n-1}\rho a'+a^{2n},
\end{equation*}
and which is positive close to $\rho=1$, of course 
there exists such functions.
\paragraph{Correction of the linear drift}
There only remains to cancel the ``linear'' term 
$$
\text{Re}\int 2\Delta^nw
\cdot\overline{\MQ\varphi_n\Delta^nz}dx
=\int 2\Delta^n\n \rho \cdot \MQ \Delta^nu dx+I_R.
$$
We apply $\Delta^n$ to the mass conservation equation,
multiply by $\Delta^n\rho$
and integrate, 
\begin{eqnarray}\nonumber
\int \Delta^{n}\big(\partial_t\rho+\text{div}(\rho u)
\big) \Delta^n\rho dx
&=&\frac{1}{2}\frac{d}{dt}\int 
(\Delta^n\rho)^2dx+\int \rho\text{div}(\Delta^nu)
\Delta^n\rho dx\\
\label{massdrift}
&=& \frac{1}{2}\frac{d}{dt}\int 
(\Delta^n\rho)^2dx-\int \MQ(\Delta^nu)
\n \Delta^n\rho dx.
\end{eqnarray}
Therefore adding \eqref{noloss} to two times 
\eqref{massdrift} we obtain
\begin{equation}\label{energyfinal}
\frac{1}{2}\frac{d}{dt}\int |\MQ\varphi_n\Delta^nz|^2+ 
|\MP\phi_n\Delta^nz|^2
+2|\Delta^n\rho|^2dx=I_R.
\end{equation}
\paragraph{Conclusion} By integration of 
\eqref{energyfinal} we find
\begin{eqnarray}\label{intenergy}
\|\MQ\varphi_n\Delta^nz\|_{L^2}^2+\|\MP \phi_n
\Delta^nz\|_{L^2}^2+2\|\Delta^nr\|_{L^2}^2\lesssim \int_0^t\|z\|_{
\mathcal{H}^{2n}}^2\|z\|_{\mathcal{W}^{1,\infty}}ds.
\end{eqnarray}
Note that for $n=0$, we have $\varphi_0=\phi_0=
\sqrt{\rho}$, therefore 
\begin{equation*}
\|z\|_{\mathcal{H}^0}\sim 
\|\MQ\varphi_0\Delta^0z\|_{L^2}^2+\|\MP \phi_0
\Delta^0z\|_{L^2}^2+2\|r\|_{L^2}^2\lesssim \int_0^t
\|z\|_{\mathcal{H}^0}^2\|z\|_{\mathcal{W}^{1,\infty}}ds,
\end{equation*}
(the estimate is actually a conservation of energy, 
see \cite{Benzoni1} paragraph $3.1$).\\
Moreover for $n\geq 1$, $\MQ \varphi_n\Delta^nz=\MQ[\varphi_n,\Delta]\Delta^{n-1}z+[\MQ\Delta,\varphi_n]
 \Delta^{n-1}z+\varphi_n\MQ\Delta^n z=R+\varphi_n 
 \MQ \Delta^nz$,  with 
 $\|R\|_{H^2n}=O(\|\rho-1\|_{W^{2,\infty}}\|z\|_{H^{2n}})$  and the same observation stands for
 $\MP\phi_n\Delta^nz$, thus
\begin{equation*}
\|\MQ\varphi_n\Delta^nz\|_{L^2}+\|\MP\phi_n\Delta^nz\|_{L^2}
=\|\Delta^nz\|_{L^2}+o(\|z\|_{H^{2n-1}}).
\end{equation*}
Using \eqref{intenergy} for $n=0,N$ we conclude
\begin{equation*}
\|z\|_{\mathcal{H}^{2N}}\sim 
\|z(t)\|_{H^{2N}}^2+\|\rho(t)-1\|_{L^2}^2\lesssim \int_0^t\|z(s)\|_{\mathcal{H}^{2N}}^2\|z(s)\|_{\mathcal{W}^{1,\infty}}ds.
\end{equation*}
\section{Control of the quadratic dispersive terms}
\label{spacetime}
The key result in \cite{AudHasp2} was the uniform bounds on $t\geq 0$
\begin{equation*}
\|xe^{-itH}\Psi\|_{L^2}\lesssim \varepsilon,\ \|\Psi\|_{W^{k,p}}
\lesssim \varepsilon/(1+t)^\gamma,
\end{equation*}
for irrotational initial data (that is $\MP u=0$). 
Actually in \cite{AudHasp2}  since $\MQ u+iw=u+iw=\n (\phi+r)$, it 
was more convenient to work on $\widetilde{\Psi}=U\phi+r$. 
This difference causes merely a shift in regularity indices as 
$\|\widetilde{\Psi}\|_{H^N\times W^{k,p}}\sim \|\Psi\|_{H^{N-1}
\times W^{k-1,p}}$. \\
We summarize here the arguments that can be used as a blackbox to obtain 
the bounds of the bootstrap argument \eqref{boot}. A few estimates are 
performed, but since the detailed analysis would be quite lengthy and is 
basically a repetition mutatis mutandis of the arguments in \cite{AudHasp2},
we choose to only sketch the argument and point to the appropriate section 
of \cite{AudHasp2} when needed.
\paragraph{Generic nonlinearity}
According to \eqref{equfinal}, and linearizing 
$a-1=a'(1)l+R$, with $R$ quadratic in $l$,
$2w-g'w=\n(2l-g'(l))=\n(g''(1)l^2/2+O(l^3))$, the quadratic 
purely dispersive nonlinearity is 
\begin{equation}\label{quaddisp}
\n\bigg((\Delta-2)b+
a'(1)l\text{div}w_1
-\frac{1}{2}\big(|\MQ u|^2-|w_1|^2\big)\bigg)
-g''(1)\n(l^2/2)
-iU^{-1}\n \text{div}\big(a'(1)l\MQ u\big).
\end{equation}
Following \cite{AudHasp2} we denote $\Psi^{\pm}$ as 
a placeholder for $\Psi$ or $\overline{\Psi}$.
Since $\MQ u=(\Psi^++\Psi^-)/2$, $w_1=
U(\Psi^+-\Psi^-)/(2i)$, $l=\Delta^{-1}U\text{div}
(\Psi^+-\Psi^-)/(2i)+R$, $R$ quadratic (see the 
change of variables of lemma \ref{formenorm}),
all quadratic nonlinearities can be written as nonlinearities 
in $\Psi^\pm$.
Their precise form does not really matter, the main 
point is that they all take the form 
\begin{equation}
\n B_g[\Psi^\pm,\Psi^\pm]=\n \mathcal{F}^{-1}\bigg(
\int_{\R^d}\widehat{\Psi^\pm}\cdot B_{g}(\eta,\xi-\eta)
\cdot \widehat{\Psi^\pm}d\eta\bigg),
\end{equation}
with $B_g$ a matrix valued symbol that can be for example
$\di \frac{(|\xi|^2+2)(a'(1)-1)}{2(2+|\eta|^2+|\xi-\eta|^2)}$, 
$\di a'(1)\frac{U(\eta)\eta}{|\eta|^2}\otimes 
(\xi-\eta)$, $\di-iU^{-1}(\xi)\frac{U(\eta)\eta^t}{|\eta|^2}
\otimes \xi$...
\paragraph{The method of space time resonances}
We denote $\widetilde{\Psi^\pm}=\mathcal{F}\big
((e^{-itH})^\pm\big)
\Psi)^\pm$. We recall that the equation 
\eqref{equfinal} reads 
$$\partial_t\Psi-iH\Psi=\mathcal{N}(\Psi,\MQ u)=
\mathcal{D}(\Psi)+\mathcal{T}(\Psi,\MP u),$$
where $\mathcal{D}$, resp. $\mathcal{T}$, correspond to the 
purely dispersive, resp. dispersive transport terms.
Let $B[\Psi^\pm,\Psi^\pm]$ a generic nonlinearity, 
the Duhamel formula leads to terms
\begin{equation}\label{bilingeneric}
\mathcal{F}\bigg(e^{-itH}\int_0^te^{i(t-s)H}
B_g[\Psi^\pm,\Psi^\pm]ds\bigg)
=\int_0^t\int_{\R^d}e^{-is\Omega_{\pm\pm}}
\widetilde{\Psi^\pm}(\eta)\cdot B_g(\eta,\xi-\eta)\cdot
\widetilde{\Psi^\pm}
(\xi-\eta)d\eta ds,
\end{equation}
where $\Omega_{\pm\pm}(\xi,\eta)=H(\xi)\mp H(\eta)\mp 
H(\xi-\eta)$. The estimates do not require to split 
the various cases $\Omega_{++},\Omega_{-,+}\cdots$,
so we write (if not ambiguous) $\Omega$ instead of 
$\Omega_{\pm,\pm}$.\\
Since $\partial_s (e^{-isH}\Psi)=e^{-isH}(
\mathcal{D}(\Psi)+\mathcal{T}(\Psi,\MP u)$, an 
integration by parts in $s$ ``improves''
the nonlinearity which becomes cubic.
Similarly from the identity 
$$
e^{-is\Omega}=\frac{\n_\eta\Omega}{-is
|\n_\eta\Omega|^2}\cdot \n_\eta e^{-is\Omega},
$$
an integration by parts in $\eta$ leads to a gain of 
decay of $1/s$. Of course these integrations by parts
are fruitful only if $\Omega$,$|\n_\eta\Omega|$ do
not cancel (resp. no time resonances and no space 
resonances), this leads to define the space-time 
resonant set as 
$\di \{(\xi,\eta):\ \Omega=0\}\cap 
\{(\xi,\eta):\ \n_\eta \Omega=0\}$.
The so-called method of 
space-time resonances simply consists in splitting the 
phase space in time non resonant and space non resonant regions and do the integration by parts
accordingly.
\\
Some difficulties are that the 
space-time resonant region is actually quite large, 
as one can check that in the case of 
$\Omega_{-+}$ it is $\{(\xi,\eta):\ \xi=0\}$, thus 
a subspace of dimension $3$ in $\R^6$. A second issue 
is that the symbol $H(\xi)=|\xi|\sqrt{2+|\xi|^2}$ 
is similar to $\sqrt{2}|\xi|$ at low frequencies 
(wave-like), so that for $\varepsilon,\eta$ small 
$\Omega_{-+}(\varepsilon\eta,\eta)\sim -3\varepsilon|\eta|^3/(2\sqrt{2})$. This 
third order cancellation is worse than for the 
Schr\"odinger equation, and prevents any use of the 
Coifman-Meyer theorem. Instead, we use the following rough 
multiplier lemma due to Guo and Pausader (inspired by lemma 
10.1 in \cite{GNT3})
\begin{lemma}[\cite{GuoPaus}] \label{multsing}
For $0\leq s\leq d/2$, let 
$\|B\|_{[B^s]}=\min\big(\|B(\eta,\xi-\eta)\|_{\widetilde{L}^\infty_\xi
 \dot{B}^s_{2,1,\eta}},\|B(\xi-\zeta,\zeta)\|_{\widetilde{L}^\infty_\xi
 \dot{B}^s_{2,1,\zeta}}\big)$. For $q_1,q_2$ such that 
 $2\leq q_2,q'_1\leq \frac{2d}{d-2s}$ and
$\di
\frac{1}{q_2}+\frac{1}{2}=\frac{1}{q_1}+\frac{1}{2}-\frac{s}{d},
$
then 
\begin{equation*}
\|B[f,g]\|_{L^{q_1}}\lesssim \|B\|_{[B^s]}\|f\|_{L^{q_2}}\|g\|_{L^2}.
\end{equation*}
Moreover, for $2\leq q_1,q_2,q_3\leq 2d(d-2s)$, and $\di 
\frac{1}{q_3}+\frac{1}{2}-\frac{s}{d}=\frac{1}{q_1}+\frac{1}{q_2}$
\begin{equation*}
\|B[f,g]\|_{L^{q_3}}\lesssim \|B\|_{[B^s]}\|f\|_{L^{q_1}}\|g\|_{L^{q_2}}.
\end{equation*}
\end{lemma}
\paragraph{The black box}
We use the following arguments directly taken from \cite{AudHasp2}: 
let $(\chi^a)_{a\in 2^\Z}$ a dyadic partition of unity, 
$\text{supp}(\chi^a)\subset\{|\xi|\sim a\}$, for $a,b,c\in (2^\Z)^3$, 
$B_g$ a symbol associated to one of the nonlinearities, 
a function $\Phi(\xi,\eta)$ 
splits the phase space in time non resonant and space 
non resonant regions (in a sense to precise in lemma \ref{estmult}),
and we define the frequency localized symbols 
\begin{equation}
B^{a,b,c,T}=\Phi \chi^a(\xi)\chi^b(\eta)\chi^c(\zeta)B_g,\ 
B^{a,b,c,X}=(1-\Phi)\chi^a\chi^b\chi^cB_g
=\chi^a\chi^b\chi^cB^X,
\end{equation}
with $\zeta=\xi-\eta$. Note that due to the relation $|\zeta|=|\xi-\eta|$, 
$B^{a,b,c}$ is zero except if $b\lesssim a\sim c,
a\lesssim b\sim c$, $c\lesssim b\sim a$.

\begin{lemma}\label{estmult}
For $a,b,c\in \Z^3$, let 
$$\mathcal{B}^{a,b,c,T}:=\frac{B^{a,b,c,T}}{\Omega},\ 
\mathcal{B}_1^{a,b,c,X}=\frac{\n_\eta \Omega B^{a,b,c,X}}
{|\n_\eta\Omega|^2},\ \mathcal{B}_2^{a,b,c,X}=\n_\eta\cdot \mathcal{B}_1^{a,b,c,X},$$ $m=\min(a,b,c),\ M=\max(a,b,c),
l=\min(b,c)$. For $0<s<2$, we have 
\begin{eqnarray*}
\text{if }M\gtrsim 1,\ \|\mathcal{B}^{a,b,c,T}\|_{[B^s]}
\lesssim \frac{\la M\ra l^{3/2-s}}{\la a \ra},& 
\|\mathcal{B}_1^{a,b,c,X}\|_{[B^s]}
\lesssim \frac{\la M\ra^2 l^{3/2-s}}{\la a \ra},\\ 
&\|\mathcal{B}_2^{a,b,c,X}\|_{[B^s]}
\lesssim \frac{\la M\ra^2 l^{1/2-s}}{\la a \ra}, 
\\
\text{if }M<<1,\ \|\mathcal{B}^{a,b,c,T}\|_{[B^s]}
\lesssim  M^{-s} l^{1/2-s},&
\|\mathcal{B}_1^{a,b,c,X}\|_{[B^s]}
\lesssim  M^{1-s} l^{3/2-s},\\
&\|\mathcal{B}_2^{a,b,c,X}\|_{[B^s]}
\lesssim  M^{-s} l^{1/2-s},
\end{eqnarray*}
\end{lemma}
\begin{lemma}\label{decaybonus}
We have for $t\geq 0$
\begin{equation}
 \|U^{-1}\Psi\|_{L^6}\lesssim \frac{1}{t^{3/5}}
 (\|xe^{-itH}\Psi\|_{L^2}+\|\Psi\|_{H^1}).
\end{equation}
\end{lemma}
\begin{proof}
By interpolation and the dispersion estimate 
\ref{dispersion}
\begin{eqnarray*}
\|U^{-1}\Psi\|_{L^6}\leq \|U^{-1/3}\Psi\|_6^{3/5}
\|U^{-2}\Psi\|_{L^6}
&\lesssim &\bigg(\frac{\|e^{-itH}\Psi\|_{L^{6/5,2}}}{t}\bigg)^{3/5}\big(\|U^{-1}\Psi\|_{L^2}+\|\Psi\|_{H^1}
\big)^{2/5}\\
&\lesssim& \bigg(\frac{\|xe^{-itH}\Psi\|_2}{t}\bigg)
^{3/5}\big(\|xe^{-itH}\Psi\|_2+\|\Psi\|_{H^1}\big)^
{2/5}.
\end{eqnarray*}

\end{proof}

\paragraph{Control of the purely dispersive quadratic terms in $W^{k,p}$}
This (long) paragraph is devoted the bootstrap 
of the $W^{k,p}$ estimate. We focus on control 
of space non resonant and time non resonant terms.
\subparagraph{Control of time non resonant terms 
in $W^{k,p}$}
Integrating by parts in $s$, the frequency localized Duhamel terms of
\eqref{bilingeneric} lead to the following quantities
\begin{equation}
I^{a,b,c,T}:=
\int_0^te^{i(t-s)H}\big(\mathcal{B}^{a,b,c,T}
[\mathcal{N}^\pm,\psi^\pm]+\mathcal{B}^{a,b,c,T}
[\psi^\pm,\mathcal{N}^\pm]\big)ds
-\big[e^{i(t-s)H}\mathcal{B}^{a,b,c,T}[\Psi^\pm,\Psi^\pm]\big]_0^{t}.
\end{equation}
Consider for example $\di \int_0^{t-1}e^{i(t-s)H}\mathcal{B}^{a,b,c,T}
[\mathcal{D}^\pm+\mathcal{T}^\pm,\Psi^\pm]ds$, 
$b\lesssim a\sim c$. We choose $p,N$ such that 
 $1/2-2\gamma>0$, $N-k-1/2+\gamma>0$ (this corresponds to 
$p$ close enough to $6$ and $N$ large enough) and apply lemma \ref{multsing}
with $s=1+\gamma$, 
\begin{eqnarray*}
\bigg\|\n^{k_1}\int_0^{t-1}\sum_{b\lesssim a\sim c}
\mathcal{B}^{a,b,c,T}\|[\mathcal{D}^\pm
,\Psi^\pm]ds\bigg\|_p
&\lesssim&
\int_0^{t-1}\sum_{b\lesssim a\sim c\leq 1}
\frac{ab\|\mathcal{B}^{a,b,c,T}\|_{[B^{1+\gamma}]}
\|U^{-1}\mathcal{D}\|_{2} \|U^{-1}\Psi\|_2}
{(t-s)^{1+\gamma}}\\
&&+\sum_{b\lesssim a\sim c,\ c\geq 1}
\frac{c^{-N+k_1}\|\mathcal{B}^{a,b,c,T}\|_{[B^{1+\gamma}]}
\|\mathcal{D}\|_{2} \|\Psi\|_{H^N}}
{(t-s)^{1+\gamma}}ds
\\
&\lesssim&
\int_0^{t-1}\sum_{b\lesssim a\sim c\leq 1}
\frac{aba^{-(1+\gamma)}
b^{1/2-(1+\gamma)}\|U^{-1}\mathcal{D}\|_{2} 
\|U^{-1}\Psi\|_2}{(t-s)^{1+\gamma}}
\\
&&+\sum_{b\lesssim a\sim c,\ c\geq 1}
\frac{b^{3/2-(1+\gamma)}
\|\mathcal{D}\|_{2} \|\Psi\|_{H^N}}
{a^{N-k_1}N(t-s)^{1+\gamma}}ds\\
&\lesssim& \int_0^{t-1}
\frac{\|U^{-1}\mathcal{D}\|_{2}(\|U^{-1}\Psi\|_2+\|\Psi\|_{H^N})}
{(t-s)^{1+\gamma}}ds,
\end{eqnarray*}
Then $\|U^{-1}\Psi\|_2\lesssim \|e^{-itH}|\Delta|^{-1/2}\Psi\|_2+\|\Psi\|_2
\lesssim \|e^{-itH}\Psi\|_{L^{6/5,2}}+\|\Psi\|_2\lesssim \|xe^{-itH}
\Psi\|_2+\|\Psi\|_2,$ where $L^{6/5,2}$ is the Lorentz space, and we used 
the generalized H\"older inequality $L^{6/5,2}\times L^{3,\infty}\subset L^2$.
On the other hand since $\n$ is in factor of all purely dispersive 
quadratic nonlinearities (see equation \eqref{quaddisp}), it compensates 
the singular factor $U^{-1}$, and one easily gets
\begin{eqnarray*}
\|U^{-1}\mathcal{D}\|_{L^2}\lesssim \|\Psi\|_{W^{2,4}}^2 \lesssim 
\|\Psi\|_{H^2}^{\frac{1+4\gamma}{4+4\gamma}}\|\Psi\|_{W^{2,p}}^{\frac{3}{4+4\gamma}}.
\end{eqnarray*}
The bootstrap assumption gives the bound
\begin{eqnarray*}
C^3\int_0^{t-1}\frac{\varepsilon^{1/2}\big(\delta+\varepsilon/(1+s)^{1+\gamma}
\big)^{\frac{3}{2(1+\gamma)}}\varepsilon}{(t-s)^{1+\gamma}}ds&\lesssim& 
C^3\varepsilon^{3/2}\bigg(\delta^{\frac{3}{2(1+\gamma)}}
+\frac{\varepsilon^{\frac{3}{2(1+\gamma)}}}{t^{\min(1+\gamma,3/(2(1+\gamma)))}}\bigg)\\
&\leq& C^3\varepsilon^{3/2}\bigg(\delta
+\frac{\varepsilon}{t^{1+\gamma}}\bigg),
\end{eqnarray*}
for $\gamma$ small enough. Like the case $d\geq 5$, the estimate of $\int_{t-1}^t$
is simpler, so is the estimate of $[e^{i(t-s)H}\mathcal{B}^{a,b,c,T}
[\Psi^\pm,\Psi^\pm]]_0^{t-1}$, and the cases $a\lesssim b\sim c$, $c\lesssim b\sim a$
are similar. More detailed computations can be found in \cite{AudHasp2} paragraph
$6.1.2$ where  the only difference is that instead of $\|\Psi\|_{W^{k,p}}
\leq C(\delta+\varepsilon/(1+t)^{1+\gamma})$, the boostrap assumption is 
$\|\Psi\|_{W^{k,p}}\leq C\varepsilon/(1+t)^{1+\gamma}$.\\
Omitting these computations, to summarize, 
\begin{equation}\label{decayD}
\bigg\|\int_0^te^{i(t-s)H}\sum_{a,b,c}\mathcal{B}^{a,b,c,T}
[\mathcal{D}^\pm,\psi^\pm]-[e^{-i(t-s)H}\mathcal{B}^{a,b,c,T}[\Psi^\pm,
\Psi^\pm]]_0^t\bigg\|_{W^{k,p}}
\lesssim C^3\varepsilon\bigg(\delta+
\frac{\varepsilon}{(1+t)^{1+\gamma}}\bigg).
\end{equation}
There remains to bound the new term 
$B[\mathcal{T}(\Psi)^\pm,\Psi^\pm]$: for $b\lesssim 
a\sim c\leq 1$
\begin{eqnarray*}
\big\|\n^k_1\int_0^{t-1}e^{i(t-s)H}\mathcal{B}^{a,b,c,T}
[\mathcal{T}(\Psi)^\pm,\Psi^\pm]ds\big\|_{L^p}
\lesssim \int_0^{t-1}\frac{a^{-\gamma}b^{1/2-\gamma}}
{(t-s)^{1+\gamma}}
\|U^{-1}\mathcal{T}\|_2\|U^{-1}\Psi\|_2ds,
\end{eqnarray*}
and for $c\geq 1$
\begin{eqnarray*}
\big\|\n^k_1\int_0^{t-1}e^{i(t-s)H}\mathcal{B}^{a,b,c,T}
[\mathcal{T}(\Psi)^\pm,\Psi^\pm]ds\big\|_{L^p}
\lesssim \int_0^{t-1}\frac{\la a\ra^kb^{3/2-(1+\gamma)}}
{(t-s)^{1+\gamma}\la a\ra^{N+1}}
\|\mathcal{T}\Psi\|_2\|\Psi\|_2ds.
\end{eqnarray*}
We deduce by summation 
\begin{eqnarray*}
\big\|\sum_{b\lesssim a\sim c}\int_0^{t-1}e^{i(t-s)H}\mathcal{B}^{a,b,c,T}
[\mathcal{T}^\pm,\Psi^\pm]ds\big\|_{W^{k,p}}
\lesssim \int_0^{t-1}\frac{\|U^{-1}\mathcal{T}\|_2
\|U^{-1}\Psi\|_2}{(t-s)^{1+\gamma}}ds\\
+\int_0^{t-1}\frac{\|\mathcal{T}\|_2
\|\Psi\|_{H^N}}{(t-s)^{1+\gamma}}ds
 \end{eqnarray*}
Unlike the purely dispersive nonlinearity, the transport-dispersive 
nonlinearity is not well-prepared, let us recall it is
\begin{equation*}
\mathcal{T}= -iU^{-1}\n (\MP u\cdot w_1)-\MQ\big(u\cdot \n \MP u
+\MP u\cdot \n\MQ u\big),
\end{equation*}
but the factor $\MP u$ is much more 
favourable thus we can simply apply the following estimates
\begin{eqnarray*}
 \|U^{-1}\mathcal{T}\|_2&\lesssim& \|\mathcal{T}\|_{W^{1,6/5}}
 \lesssim \|\MP u\|_{W^{k,6/5}}(\|\Psi\|_{H^N}+\|\MP u\|_{H^N}),\\
 \|\mathcal{T}\|_2&\lesssim& \|\MP u\|_2(\|\Psi\|_{H^N}+\|\MP u\|_{H^N }),
\end{eqnarray*}
combined with the bootstrap assumption \eqref{boot} we find
\begin{equation}\label{decayT1}
 \big\|\sum_{b\lesssim a\sim c}\int_0^{t-1}e^{i(t-s)H}\mathcal{B}^{a,b,c,T}
[\mathcal{T}^\pm,\Psi^\pm]ds\big\|_{W^{k,p}}
\lesssim C^3\int_0^{t-1}\frac{\varepsilon^2\delta}{(t-s)^{1+\gamma}}ds
\lesssim C^3\varepsilon^2\delta.
\end{equation}
The integral over $[t-1,t]$ is estimated in the same spirit : for $c\leq 1$
\begin{eqnarray*}
 \big\|\int_{t-1}^te^{i(t-s)H}\mathcal{B}^{a,b,c,T}[\mathcal{T}^\pm,
 \Psi^\pm]ds\big\|_{W^{k,p}}
 &\lesssim &\int_{t-1}^t\|
 \mathcal{B}^{a,b,c,T}[\mathcal{T}^\pm, \Psi^\pm]\big\|_{H^{k+2}}ds
 \\
 &\lesssim& \int_{t-1}^tab\|\mathcal{B}^{a,b,c,T}\|_{[B^{1+\gamma}]}
 \|U^{-1}\mathcal{T}\|_2\|U^{-1}\Psi\|_2ds,
\end{eqnarray*}
 for $c\geq 1$ 
 \begin{eqnarray*}
 \big\|\int_{t-1}^te^{i(t-s)H}\mathcal{B}^{a,b,c,T}[\mathcal{T}^\pm,
 \Psi^\pm]ds\big\|_{W^{k,p}}
 &\lesssim &\int_{t-1}^t\|
 \mathcal{B}^{a,b,c,T}[\mathcal{T}^\pm, \Psi^\pm]\big\|_{H^{k+2}}ds
 \\
 &\lesssim& \int_{t-1}^t\frac{\|\mathcal{B}^{a,b,c,T}\|_{[B^{1+\gamma}]}}
 {c^{N-k-2}}
 \|\mathcal{T}\|_2\|\Psi\|_{H^N}ds.
\end{eqnarray*}
As previously, we have $\|U^{-1}\mathcal{T}\|_2+\|\mathcal{T}\|_2\lesssim 
\big(\|\MP u\|_{W^{k,6/5}}  +\|\MP u\|_{H^k}\big)\big(\|\MP u\|_{H^N}
+\|\Psi\|_{H^N}\big)$
and $\sum_{b\lesssim a\sim c\leq 1}ab\mathcal{B}^{a,b,c,T}\|
_{[B^{1+\gamma}]}+\sum_{b\lesssim a\sim c,\ c\geq 1}
\frac{\mathcal{B}^{a,b,c,T}\|_{[B^{1+\gamma}]}}{c^{N-k-2}}<\infty$, thus
\begin{equation}\label{decayT2}
\big\|\int_{t-1}^t\sum_{b\lesssim a\sim c}e^{i(t-s)H}\mathcal{B}^{a,b,c,T}
[\mathcal{T}^\pm,\Psi^\pm]ds\big\|_{W^{k,p}}\lesssim C^3\int_{t-1}^t
\varepsilon^2\delta ds=C^3\varepsilon^2\delta.
\end{equation}
Putting together \eqref{decayD},\eqref{decayT1},\eqref{decayT2},
\begin{equation}
\label{decayresT}
\big\|\sum_{a,b,c}I^{a,b,c,T} \|_{W^{k,p}}\lesssim 
C^3\varepsilon\bigg(\delta+\frac{\varepsilon}
{(1+t)^{1+\gamma}}\bigg).
\end{equation}
\subparagraph{Control of $I^{a,b,c,X}$ in $W^{k,p}$} Integration by parts 
in $\eta$ does not require to handle the nonlinear term $\mathcal{T}$ which,
in this appendix, is the only novelty compared to \cite{AudHasp2}. \\
If necessary, we simply reproduce the argument with the minor modifications:\\
Since control for $t$ small just follows from the $H^N$ bounds, we focus on $t\geq 1$, and 
the integral over $[1,t-1]$.\vspace{3mm}\\
\textbf{Frequency splitting}\vspace{2mm}\\
Since we only control $xe^{-itH}z$ in $L^\infty L^2$, in order to handle the loss of derivatives 
we follow the idea from \cite{GMS2} which corresponds to distinguish
 low and high frequencies with a threshold frequency depending on $t$. Let $\theta\in C_c^\infty (\R^+)$, 
$\theta|_{[0,1]}=1,\ 
\text{supp}(\theta)\subset[0,2]$, $\Theta(t)=\theta(\frac{|D|}{t^\nu})$, $\nu>0$ small to choose 
later. For any quadratic term 
$B_g[z,z]$, we write
\begin{equation*}
B_g[z^\pm,z^\pm]=\overbrace{B_g[(1-\Theta(t))z^\pm,z^\pm]
+B_g[\Theta(t)z^\pm,(1-\Theta)(t)z^\pm]}
^{\text{high frequencies}}
+\overbrace{B_g[\Theta(t)z^\pm,\Theta(t)z^\pm]}^{\text{low frequencies}}.
\end{equation*}
\subsubsection*{High frequencies}
Using the dispersion estimate \ref{dispersion}, 
product estimates and Sobolev embedding we have for 
$\frac{1}{p_1}=\frac{1+\gamma}{3}$ and for any quadratic term $B^X[\Psi^\pm,\Psi^\pm]$:
\begin{equation}
\begin{aligned}
&\bigg\|\int^{t-1}_{1}e^{i(t-s)H}\big(UB_g[(1-\Theta(t))\Psi^\pm,\Psi^\pm]
+U B_j[\Theta(t)z,(1-\Theta)(t)\Psi^\pm]\big)ds\bigg\|_{W^{k,p}}\\
&\leq  \int^{t-1}_{1}\frac{1}{(t-s)^{1+\gamma}}\|\Psi\|_{W^{k+2,p_1}}  \|(1-\Theta(s))\Psi\|_{H^{k+2}}ds\\
&\leq  \int^{t-1}_{1}\frac{1}{(t-s)^{1+\gamma}}\|\Psi\|_{H^N}^2\frac{1}{s^{\nu(N-2-k)}}ds,
\end{aligned}
\end{equation}
choosing $N$ large enough so that $\nu(N-2-k)\geq 1+\gamma$, we obtain a bound $C_1C^2\varepsilon^2/
t^{1+\gamma}$. 
\subsubsection*{Low frequencies}
We estimate now quadratic term of the form 
$B^{a,b,c,X}[\Theta\Psi^\pm,\Theta\Psi^\pm]$ wich leads to consider:
$${\cal F}I^{a,b,c,X}_3=e^{itH(\xi)}
\int_1^{t-1}\int_{\R^N} \bigg(( e^{-is\Omega} B^{a,b,c,X}(\eta,\xi-\eta)
\widetilde{\Theta \Psi^\pm}(s,\eta)\widetilde{\Theta \Psi^\pm}(s,\xi-\eta)\bigg)d\eta ds,$$
with  $\Omega=H(\xi)\mp H(\eta)\mp H(\xi-\eta)$. Using  $\displaystyle e^{-is\Omega}=\frac{i\nabla_{\eta}\Omega}{s|\nabla_{\eta}\Omega|^2}
\cdot \nabla_{\eta}e^{-is\Omega}$ and denoting $Ri=\frac{\n}{|\n|}$ 
the Riesz operator, $\Theta'(t):=\theta'(\frac{|D|}{t^\delta})$, $J=e^{itH}xe^{-itH}$, an integration by 
part in $\eta$ gives:
\begin{equation}
\begin{aligned}
I^{a,b,c,X}_3=
&-{\cal F}^{-1}(e^{itH(\xi)}\biggl(
\int_1^{t-1}\frac{1}{s} \int_{\R^N}  \big(e^{-is \Omega(\xi,\eta)} \mathcal{B}^{a,b,c,X}_{1,j}(\eta,\xi-\eta)
\cdot \n_\eta [\Theta\widetilde{\Psi^\pm}(\eta)\Theta\widetilde{\Psi^\pm}(\xi-\eta)]\\
&\hspace{65mm}+\mathcal{B}^{a,b,c,X}_{2,j}(\eta,\xi-\eta)\widetilde{\Theta \Psi^\pm}(\eta)
\widetilde{\Theta \Psi^\pm}(\xi-\eta)d\eta\big)ds\biggl)\\
=&
-\int_1^{t-1}\frac{1}{s}e^{i(t-s)H}\bigg(\mathcal{B}^{a,b,c,X}_{1,j}
[\Theta(s)(Jz)^\pm,\Theta(s)\Psi^\pm]
-\mathcal{B}^{a,b,c,X}_{1,j}[\Theta(s)\Psi^\pm,\Theta(s)(Jz)^\pm]\\
&\hspace{8cm}
+\mathcal{B}^{a,b,c,X}_{2,j}[\Theta(s)\Psi^\pm,\Theta(s)
\Psi^\pm]\bigg)ds\\
&-\int_1^{t-1} \frac{1}{s}e^{i(t-s)H}\bigg(\mathcal{B}^{a,b,c,X}_{1,j}
[\frac{1}{s^\delta}Ri\,\Theta'(s)\Psi^\pm,\Theta(s) \Psi^\pm]\\
&\hspace{3.5cm}-\mathcal{B}^{a,b,c,X}_{1,j}[\Theta(s)\Psi^\pm,\frac{1}{s^\delta}Ri
\Theta'(s)\Psi^\pm]\bigg)ds.
\end{aligned}
\label{ippespace}
\end{equation}
 where we recall:
\begin{equation*}
\displaystyle\mathcal{B}^{a,b,c,X}_{1}=\frac{ \nabla_{\eta}\Omega}{|\nabla_{\eta}
\Omega|^2}B^{a,b,c,X},\
\displaystyle\mathcal{B}^{a,b,c,X}_{2}=\nabla_{\eta}\cdot\mathcal{B}_{1}^{a,b,c,X}.
\end{equation*}
\noindent We now use these estimates to bound the first term of \eqref{ippespace}. There are
three areas to consider: $b\lesssim c\sim a,\ c\lesssim c\lesssim a\sim b,\ a\lesssim b\sim c$.
\subparagraph{Estimates for quadratic terms involving $\mathcal{B}^{a,b,c,X}_1$}  In the case $c\lesssim a\sim b$, let $\varepsilon_1>0$ to be 
fixed later. Using Minkowski's inequality, 
dispersion and the rough multiplier theorem \ref{multsing} with $s=1+\ve_1$, $\frac{1}{q}
=1/2+(\gamma-\varepsilon_1)/3$, $s=4/3$, 
$\frac{1}{q_1}=7/18+\gamma/3$ for $a\geq 1$, $0\leq k_1\leq k$ we obtain
$$
\begin{aligned}
&\big\|\n^{k_1}\int_1^{t-1}\frac{1}{s}e^{i(t-s)H}\sum_{c\lesssim a\sim b}
\mathcal{B}^{a,b,c,X}_{1}[\Theta(s)(Jz)^\pm,\Theta(s)z^\pm]ds\big\|_{L^{p}}\\
&\lesssim \int_1^{t-1}\frac{1}{s(t-s)^{1+3\ve}}
\sum_{c\lesssim a\sim b\leq 1}
\|\mathcal{B}^{a,b,c,X}_{1}\|_{[B^{1+\ve_1}]}\|\Theta(s)Jz\|_{L^2}
\|\Theta(s)z]\|_{L^{q}}\\
&\hspace{13mm}+\sum_{c\lesssim a\sim b,\ 
1\lesssim a\lesssim s^\nu}a^{k}
\|\mathcal{B}^{a,b,c,X}_{1}\|_{[B^{4/3}]}\|\Theta(s)Jz\|_{L^{2}}
\|\Theta(s)z]\|_{L^{q_1}}\big)ds\\[2mm]
&\lesssim \int_1^{t-1}\frac{1}{s(t-s)^{1+3\ve}}\bigg(\sum_{a\lesssim 1}\sum_{c\lesssim a\sim b}
\|\mathcal{B}^{a,b,c,X}_{1}\|_{[B^{1+\ve_1}]}\|
\Theta(s)Jz\|_{L^2}\|\Theta(s)z]\|_{L^{q}}\\
&\hspace{27mm}+\sum_{1\lesssim a \lesssim s^{\nu}}a^{k}\sum_{c\lesssim a\sim b}
\|\mathcal{B}^{a,b,c,X}_{1}\|_{[B^{4/3}]}\|\Theta(s)Jz\|_{L^{2}}
\|\Theta(s)z]\|_{L^{q_1}}\bigg)ds
\end{aligned}
$$
Using lemma \ref{estmult} and interpolation we have for $\ve_1<1/4$ and 
$\varepsilon_1-\gamma>0$,
$$
\begin{aligned}
&\sum_{a\lesssim 1}\sum_{c\lesssim a\sim b}
\|\mathcal{B}^{a,b,c,X}_{1}\|_{[B^{1+\ve_1}]}\lesssim 
\sum_{a\lesssim 1}a^{1-(1+\ve_1)}\sum_{c\lesssim a}c^{\frac{3}{2}-(1+\ve_1)}\lesssim 1,\\
&\|\psi(s)\|_{L^{q}}\lesssim\|\psi(s)\|^{\frac{\e_1-\gamma}{1+\gamma}}_{L^{p}}\|\psi(s)\|^{1-\frac{\e_1-\gamma}{1+\gamma}}_{L^{2}} 
\lesssim  C \varepsilon^{1-\frac{\e_1-\gamma}{1+\gamma}}
\bigg(\delta+\frac{\varepsilon}{(1+s)^{1+\gamma}}\bigg)^{\frac{\varepsilon_1-\gamma}{1+\gamma}}.
\end{aligned}
$$ 
In high frequencies we have:
$$
\begin{aligned}
&\sum_{1\lesssim a \lesssim s^{\nu}}a^{k} \sum_{c\lesssim a\sim b} 
\frac{\la M\ra^2c^{3/2-4/3}}{\la a\ra}
\lesssim s^{\nu(k+7/6)},\
\|\psi(s)\|_{L^{q_1}}\lesssim \varepsilon^{\frac{2+6\gamma}{3+3\gamma}}
\bigg(\delta+\frac{\varepsilon}{(1+s)^{1+\gamma}}\bigg)^{\frac{1-3\gamma}{3+3\gamma}}.
\end{aligned}
$$
Finally we conclude that if $\min\big(\varepsilon_1-2\gamma,1/3-2\gamma
-\nu(k+7/6)\big)\geq 0$ (this choice is possible provided $\gamma$ and $\nu$ are 
small enough):
\begin{equation}
\begin{aligned}
\big\|\int_1^{t-1}\frac{1}{s}e^{-i(t-s)H}&\sum_{c\lesssim a\sim b}\mathcal{B}^{a,b,c,X}_1
[\Theta(s)(Jz)^\pm,\Theta(s)z^\pm]ds\big\|_{W^{k,p}}\\
&\lesssim \int_1^{t-1}\frac{C^2\varepsilon}{s(t-s)^{1+\gamma}}
\varepsilon^{1-\frac{\ve_1-\gamma}{1+\gamma}}
\bigg(\delta+\frac{\varepsilon}{(1+s)^{1+\gamma}}\bigg)^{\frac{\varepsilon_1-\gamma}{1+\gamma}}\\
&\hspace{2cm}
+\frac{C^2\varepsilon s^{\nu(7k+6)}}{s(t-s)^{1+\gamma}}\varepsilon^{\frac{2+6\gamma}{3+3\gamma}}
\bigg(\delta+\frac{\varepsilon}{(1+s)^{1+\gamma}}\bigg)^{\frac{1-3\gamma}{3+3\gamma}}ds\\
&\lesssim \frac{C^2\varepsilon^2}{t^{1+\gamma}}+\frac{C^2
\varepsilon^{2-\frac{\ve_1-\gamma}{1+\gamma}}\delta^{\frac{\varepsilon_1-\gamma}{1+\gamma}}}{t}
+\frac{C^2\varepsilon^{\frac{5+9\gamma}{3+3\gamma}}\delta^{\frac{1-3\gamma}{3+3\gamma}}}{t^{1-\nu(7k+6)}}\\
&\lesssim \frac{C^2\varepsilon^2}{t^{1+\gamma}}
+\frac{C^2\varepsilon^2}{t^{\frac{1+\gamma}{1+2\gamma-\varepsilon_1}}}
+C^2\varepsilon\delta
+\frac{C^2\varepsilon^2}{t^{(1-\nu(7k+6))\frac{3+3\gamma}{2+6\gamma}}}
\\
&\lesssim \frac{C^2\varepsilon^2}{t^{1+\gamma}}+C^2\varepsilon\delta.
\end{aligned}\label{estimB1}
\end{equation}

The cases $b\lesssim c\sim a$ are very similar.
The term $\displaystyle \n^{k_1}\int_1^{t-1}\frac{1}{s}e^{i(t-s)H}\mathcal{B}^{a,b,c,X}_{1}
[\Theta(s)z^\pm,\Theta(s)(Jz)^\pm]ds$ is symmetric  while the terms
$$
\begin{aligned}
&\|\n^{k_1}\int_1^{t-1} \frac{1}{s}e^{i(t-s)H}\big(\mathcal{B}^{a,b,c,X}_{1}
[\frac{1}{s^\delta}\frac{\n}{|\n|}\Theta'(s)z^\pm,\Theta(s)z^\pm]\\
&\hspace{3cm}-\mathcal{B}^{a,b,c,X}_{1}[\Theta(s)z^\pm,
\frac{1}{s^\delta}\frac{\n}{|\n|}\Theta'(s)z^\pm]\big)ds\|_{L^p},
\end{aligned}
$$
are simpler since there is no weighted term $Jz$ involved.
\subparagraph{Estimates for quadratic terms involving $\mathcal{B}^{a,b,c,X}_2$}
The last term to consider is 
$$\big\|\n^{k_1}\int_1^{t-1}\frac{1}{s}e^{i(t-s)H}\sum_{a,b,c}\mathcal{B}^{a,b,c,X}_{2}
[\Theta(s)z^\pm,\Theta(s)z^\pm]ds\big\|_{L^p}.  $$
Let us focus on the case $b\lesssim a\sim c$. We use the same indices as for 
$\mathcal{B}_1^{a,b,c,X}$: $s=1+\varepsilon_1$,
$\frac{1}{q}=1/2+(\gamma-\ve_1)/3$, $\frac{1}{q_1}=7/18+\gamma/3$,
\begin{equation}\label{estimX}
\begin{aligned}
&\big\|\n^{k_1}\int_1^{t-1}\frac{1}{s}e^{i(t-s)H}\sum_{b\lesssim a\sim c}
\mathcal{B}^{a,b,c,X}_{2}
[\Theta(s)\Psi^\pm,\Theta(s)\Psi^\pm]ds\big\|_{L^p}\\
&\lesssim \int_1^{t-1}\frac{1}{s(t-s)^{1+\gamma}}\bigg(\sum_{a\leq 1}
\sum_{b\lesssim a\sim c}
U(b)U(c)\|\mathcal{B}^{a,b,c,X}_{2}\|_{[B^{1+\ve_1}]}\|U^{-1}\Theta(s)\Psi
\|_{L^2}\|U^{-1}\Theta(s)\Psi]\|_{L^q}\\
&\hspace{3cm}
+\sum_{1\leq a \lesssim s^{\nu}}a^{k}\sum_{b\lesssim a\sim c}\frac{U(b)}{\la c\ra^k}
\|\mathcal{B}^{a,b,c,X}_{2}\|_{[B^{4/3}]}\|U^{-1}\Theta(s)\Psi\|_{L^{2}}
\|\la \n\ra^k\Theta(s)\Psi]\|_{L^{q_1}}\bigg)ds
\end{aligned}
\end{equation}
According to lemma \ref{estmult}, we have for the first sum (provided 
$\varepsilon_1<1/4$):
$$
\sum_{a\leq 1}\sum_{b\lesssim c\sim a}U(b)U(c)\|\mathcal{B}_2^{a,b,c,X}\|
 _{[B^{1+\varepsilon_1}]}\lesssim 
\sum_{a\leq 1}\sum_{b\lesssim c\sim a}b^{1/2-\varepsilon_1}
a^{\varepsilon_1}\lesssim 1.
$$ 
and according to proposition \ref{decaybonus}
and the bootstrap assumption \eqref{boot}
\begin{eqnarray*}
\|U^{-1}\Psi(s)\|_{L^2}&\lesssim&\|\psi\|_{X},
\\
\|U^{-1}\Psi(s)\|_{L^{q}}&\lesssim& \|U^{-1}\Psi\|_{L^2}^{1-\ve_1+\gamma}
\|U^{-1}\Psi\|_{L^6}^{\ve_1-\gamma}\\
&\lesssim& \frac{\|xe^{-itH}\Psi\|_{2}^{1-\varepsilon_1+\gamma}
\big(\|xe^{-itH}\Psi\|_2+
\|\Psi\|_{H^1}\big)^{\varepsilon_1-\gamma}}
{s^{\frac{3(\ve_1-\gamma)}{5}}}
\\
&\lesssim& \frac{C\varepsilon}{s^{\frac{3(\ve_1-\gamma)}{5}}}.
\end{eqnarray*}
Now for $M\gtrsim 1$ 
\begin{eqnarray*}
\sum_{1\leq a \lesssim s^{\nu}}a^{k}\sum_{b\lesssim c\sim a}
\frac{ U(b) \la M\ra^2b^{1/2-4/3}}{\la a\ra \la c\ra^k}\lesssim \sum_{1\leq a\lesssim s^\nu}
a \lesssim s^{\nu},\hspace{0.4cm} 
\|\psi(s)\|_{W^{k,q_1}}\lesssim \varepsilon^{\frac{2+6\gamma}{3+3\gamma}}
\bigg(\delta+\frac{\varepsilon}{(1+s)^{1+\gamma}}\bigg)^{\frac{1-3\gamma}{3+3\gamma}}.
\end{eqnarray*}
We inject these estimates in \eqref{estimX} and from 
the same computations as for \eqref{estimB1}
we find that if 
$\min\big(3(\varepsilon_1-\gamma)/5,1/3-\gamma
-\nu\big)\geq \gamma$,
\begin{equation}\label{estimB2}
\begin{aligned}
\big\|\int_1^{t-1}\frac{1}{s}e^{i(t-s)H}\sum_{b\lesssim c\sim a}&
\mathcal{B}^{a,b,c,X}_{2}
[\Theta(s)\Psi^\pm,\Theta(s)\Psi^\pm]ds\big\|_{W^{k,p}}\\
&\lesssim \int_1^{t-1}\frac{C^2\varepsilon^2}
{(t-s)^{1+\gamma}s^{1+\frac{3(\varepsilon_1-\gamma)}{5}}}+\frac{C^2\varepsilon^{1+\frac{2+6\gamma}{3+3\gamma}}}{(t-s)^{1+\gamma}}
\bigg(\delta+\frac{\varepsilon}{(t-s)^{1+\gamma}}
\bigg)^{\frac{1-3\gamma}{3+3\gamma}}ds
\\
&\lesssim \frac{C^2\varepsilon^2}{t^{1+\gamma}}+
C^2\varepsilon\delta.
\end{aligned}
\end{equation}
The two other cases $c\lesssim a\sim b$ and $a\lesssim b\sim c$ can be treated in a similar way.\\
From \eqref{estimB1},\eqref{estimB2}
\begin{equation}\label{decayresX}
\big\|\sum_{a,b,c}I^{a,b,c,X}\big\|_{W^{k,p}}
\lesssim C^2\varepsilon
\bigg(\delta+\frac{\varepsilon}{(1+t)^{1+\gamma}}\bigg)
\end{equation}
\subparagraph{Conclusion}
From \eqref{decayresT} and \eqref{decayresX}, 
we have
\begin{equation}\label{decayDfinal}
 \big\|\int_0^1e^{i(t-s)H}\mathcal{D}(\Psi)ds\big\|
 _{W^{k,p}}\lesssim C^2\varepsilon\bigg(\delta+
 \frac{\varepsilon}{(1+t)^{1+\gamma}}\bigg).
\end{equation}
Higher order (cubic and quartic) terms are easier to control, we refer to \cite{AudHasp2} paragraph $5.2$,
we conclude
\begin{equation}\label{decaytot}
 \bigg\|e^{itH}\Psi_0\|+\int_0^te^{i(t-s)H}\mathcal{N}
 (\Psi,\MP u)ds\bigg\|_{W^{k,p}}\leq 
 \frac{C_1\varepsilon}{(1+t)^{1+\gamma}}
 +C_1C^2\varepsilon\bigg(\delta+\frac{\varepsilon}{(1+t)
 ^{1+\gamma}},\bigg).
\end{equation}
so that choosing $C$ large enough, $\varepsilon$ small 
enough we have as expected 
$$
\|\Psi\|_{W^{k,p}}
\leq \frac{C}{2}(\delta+\varepsilon/(1+t)^{1+\gamma}).
$$
\paragraph{Control of the purely quadratic terms 
in the weighted norm}
We refer to the paragraph $6.2$ in \cite{AudHasp2}, 
which can be applied with the same ``routine'' 
modifications as for the $W^{k,p}$ estimates.

\bibliography{biblio}
 \bibliographystyle{plain}
\end{document}